\documentclass[sn-nature,Numbered]{sn-jnl}%
\usepackage{graphicx}%
\usepackage{multirow}%
\usepackage{amsmath,amssymb,amsfonts}%
\usepackage{amsthm}%
\usepackage{mathrsfs}%
\usepackage[title]{appendix}%
\usepackage[table,svgnames]{xcolor}
\usepackage{textcomp}%
\usepackage{manyfoot}%
\usepackage{booktabs}%
\usepackage{algorithm}%
\usepackage{algorithmicx}%
\usepackage{algpseudocode}%
\usepackage{listings}%
\usepackage{paralist}%
\usepackage{makecell}%
\usepackage{comment}%
\usepackage{multirow}
\usepackage{todonotes}
\usepackage{soul}
 
\makeatletter
\def\blfootnote{\xdef\@thefnmark{}\@footnotetext}
\makeatother

\newtheorem{theorem}{Theorem}
\newtheorem{proposition}[theorem]{Proposition}%
\newtheorem{lemma}[theorem]{Lemma}
\newtheorem{corollary}[theorem]{Corollary}

\newtheorem{remark}{Remark}%
\newtheorem{assumption}{Assumption}

\title[Convex SIP algorithms with inexact separation oracles]{Convex semi-infinite programming algorithms with inexact separation oracles\blfootnote{{\textit{*Corresponding author}}\\\textit{Email addresses:} \texttt{antoine.oustry@polytechnique.org} (Antoine Oustry), \texttt{mcerulli@unisa.it} (Martina Cerulli)}}

\author[1,2]{Antoine Oustry}
\author[3,*]{Martina Cerulli}

\affil[1]{\orgdiv{LIX}, \orgname{Institut Polytechnique de Paris}, \orgaddress{\city{Palaiseau}, \postcode{91120}, \country{France}}}

\affil[2]{\orgname{Ecole des ponts}, \orgaddress{\city{Marne-la-Vallée}, \country{France}}}

\affil[3]{\orgdiv{Department of Computer Science}, \orgname{University of Salerno}, \orgaddress{\postcode{84084}, \country{Italy}}\vspace*{-1em}}

\begin{document}
\maketitle
\vspace*{-1cm}
\fcolorbox{red}{white}{\parbox{\textwidth}{This paper has been accepted for publication in \textbf{Optimization Letters} (Vol. 19, No. 3, Pages 437--462, 2025). The \textbf{final published version} is available at \url{https://doi.org/10.1007/s11590-024-02148-3}.}}
\vspace*{4mm}

\abstract{Solving convex Semi-Infinite Programming (SIP) problems is challenging when the separation problem, namely, the problem of finding the most violated constraint, is computationally hard.
We propose to tackle this difficulty by solving the separation problem approximately, i.e., by using an inexact oracle. Our focus lies in two algorithms for SIP, namely the Cutting-Planes (CP) and the Inner-Outer Approximation (IOA) algorithms. We prove the CP convergence rate to be in $O(1/k)$, where $k$ is the number of calls to the limited-accuracy oracle, if the objective function is strongly convex. Compared to the CP algorithm, the advantage of the IOA algorithm is the feasibility of its iterates. In the case of a semi-infinite program with a Quadratically Constrained Quadratic Programming separation problem, we prove the convergence of the IOA algorithm toward an optimal solution of the SIP problem despite the oracle's inexactness.}

\keywords{Semi-Infinite Programming, Inexact Oracle, Separation problem\vspace*{-1.8em}}



\maketitle

\section{Introduction}

Standard Semi-Infinite Programming (SIP) problems are optimization problems with a finite number of variables and infinitely many constraints. These constraints are indexed by a continuous parameter that takes values from a certain parameter set. Given two integers $m,n \in \mathbb{N}^+$, we consider two continuous functions $F \colon \mathbb{R}^m \to \mathbb{R}$ and $G \colon \mathbb{R}^m \times \mathbb{R}^n \to \mathbb{R}$. Furthermore, we consider two non-empty and compact sets $\mathcal{X} \subset \mathbb{R}^m$, and $\mathcal{Y} \subset \mathbb{R}^n$. The SIP problem on which we focus is formulated as\vspace{-0.8em}
\begin{align}
  \left\{\begin{array}{rl}\vspace*{-1mm}
    \min\limits_{x\in \mathcal{X}} & F(x)  \\
    \text{s.t.} & G(x,y) \leq 0 \quad \forall y \in \mathcal{Y},
    \label{eq:SIP}
    \tag{\mbox{$\mathsf{SIP}$}}
  \end{array}\right.
\end{align}
with the value function $\phi(x)$ defined as $\phi(x) =  \max_{y \in \mathcal{Y}}  G(x,y),$ also referred to as \textit{lower-level problem} or \textit{separation problem} in the rest of the paper.
The infinitely many constraints $ G(x,y) \leq 0, \,\forall y \in \mathcal{Y}$ in \eqref{eq:SIP} can be reformulated as $\phi(x) \leq 0$, where function~$\phi$ is continuous as a direct application of the Maximum Theorem \cite[Th.~2.1.6]{aubin_jean-pierre_viability_1991}. 
This paper restricts to the following convex setting for the formulation \eqref{eq:SIP}.
\begin{assumption}
The non-empty and compact set $\mathcal{X}$ is convex, the function $F(x)$ is convex and the function $G(\cdot,y)$ is linear for every $y \in \mathcal{Y}$.
\label{as:convex}
\end{assumption}
We underline that we make no convexity assumptions regarding the set $\mathcal{Y}$ or the function $G(x, \cdot)$. 
This setting covers the case where the function $G(\cdot,y)$ is affine for every $y \in \mathcal{Y}$. Because of the infinite number of constraints, SIP problems are challenging optimization problems \cite{hettich1983review,stein2012,djelassi}, for the solution of which several methods have been developed in the literature. Whenever the inner problem is convex and regular, it can be replaced by its KKT first-order optimality conditions \cite{stein2003interior,floudas2008adaptive,goberna1998}, obtaining a problem with complementarity constraints.
However, this approach is not applicable in the general case of nonconvex lower-level problems. A valid alternative is the Iterative Discretization methods \cite{reemtsen1991discretization,hettich1986implementation,still2001discretization,Schwientek2021}, consisting in replacing the infinite constraints with several finite constraints, by approximating the infinite parameter set with a finite subset. The original SIP problem can then be transformed into a sequence of finite optimization problems solved using standard optimization techniques. By refining the discretization and solving the finite optimization problems iteratively, the optimal solution to the original SIP problem can be obtained. The convergence rate of the error between the SIP solution and the solution of each discretized program depends on the solution's order and the choice of the gridpoints \cite{still2001discretization}.
The classical Blankenship and Falk's algorithm \cite{blankenship1976infinitely,betro2004accelerated,tichatschke1988cutting} is a discretization method where the considered finite subset of constraints is increased at each iteration by adding the most violated constraint. In the case of convex SIP problems, it corresponds to Kelley's algorithm \cite{kelley1960cutting} also known as Cutting-Planes (CP) algorithm. In the exchange methods \cite{goberna1998,zhang2010new}, at every iteration, some new constraints are added and some old constraints may be deleted.
Recently, in \cite{seidel2022}, a new adaptive discretization method based on Blankenship and Falk CP \cite{blankenship1976infinitely} is proposed, exhibiting a \textit{local} quadratic rate of convergence to a SIP stationary point.  
In \cite{cerulli2022}, a convergent Inner-Outer Approximation (IOA) algorithm is introduced to solve convex SIP problems by combining a CP and a lower-level dualization approach. 
In these methods, in order to address the infinite number of constraints, an optimization algorithm is employed to solve problem~$\phi(x)$ for specific values of $x$. However, this may be computationally difficult, and assuming we solve it exactly is not necessarily realistic. Therefore, we may assume that an \textit{inexact separation oracle} (i.e., a black-box algorithm) is used to compute a feasible point of the problem~$\phi(x)$ with a relative optimality gap of $\delta \in [0,1)$. More precisely, for every $x \in \mathcal{X},$ this ``$\delta$-oracle'' computes $\hat{y}(x) \in \mathcal{Y}$, and an upper bound $\hat{v}(x) \geq \phi(x)$ such that $
 \hat{v}(x)- G(x,\hat{y}(x)) \leq \delta \: \lvert \phi(x) \rvert.$
Consequently, the following inequalities hold:\vspace{-0.5em}
\begin{align}
\phi(x) -\delta \: \lvert \phi(x) \rvert \leq  G(x,\hat{y}(x)) \leq   \phi(x) \leq \hat{v}(x) \leq \phi(x) +\delta \: \lvert \phi(x)\rvert.
 \label{eq:oraclebound}
\end{align}

The literature on convex optimization algorithms using an inexact oracle is extensive. In \cite{daspremont2008}, the Fast Gradient Method proposed in \cite{nesterov2005} is extended to include inexact gradient computation.  In \cite{devolder2014first}, the general concept of inexact oracle for convex problems is introduced and applied to First-Order methods, i.e., primal, dual, and fast gradient methods. The classical primal-dual gradient method can be seen as slow but robust with respect to (w.r.t.) oracle errors. The fast one, instead, is faster but sensitive w.r.t.\ oracle error. In \cite{devolder2013intermediate}, the same authors propose the intermediate gradient method, combining classical and fast gradient methods, providing the flexibility to select a suitable parameter value that balances the convergence rate and the accumulation of oracle errors. Both stochastic and deterministic errors are considered in the oracle information in \cite{dvurechensky2016stochastic}.
In the literature on convex minmax problems (which can be seen as a particular type of SIP problem), the oracle's inexactness is rarely considered. The approximate calculation of the optimal solution of the inner problem has been explored in \cite{gaudioso2006,gaudioso2009,fuduli2015}, in the context of bundle algorithms. When dealing with nonsmooth convex SIP problems, another bundle method with inexact oracle solving the inner problem is proposed in \cite{pang2016constrained}. For general SIP programs, the algorithms proposed in \cite{mitsos2011global,djelassi2017} allow for suboptimal solution of the involved subproblems, i.e., the lower-level, the upper and lower bounding problems. Indeed, in \cite{mitsos2011global}, within the proposed method, solving restrictions of the discretization-based relaxations, the assumption regarding the solutions of the subproblems is only that their accuracy suffices to definitively determine the feasibility of a given iterate.
Similarly, in \cite{djelassi2017}, where the algorithms proposed in \cite{mitsos2011global} and \cite{tsoukalas2011} are combined, subproblems are assumed to be solvable to an arbitrary absolute optimality tolerance, which is adjusted as necessary according to a refinement scheme. Two other adaptive discretization algorithms, still leveraging the same restriction idea of \cite{mitsos2011global}, are proposed in \cite{schmid2022approximate} for convex SIP problems, with finite termination guarantees for any arbitrary precision, despite the possibility, again, of solving the inner problems only approximately.

Compared to these earlier works, the contribution of this paper is to (i) prove a global rate of convergence for the Blankenship and Falk's CP algorithm \cite{blankenship1976infinitely}, despite the inexactness of the separation oracle, in the context of a strongly convex objective function (ii) show that the IOA algorithm is able, once again despite the inexactness of the oracle, to generate a sequence of feasible points converging to an optimum. In other words, this paper extends the convergence results for the CP and IOA algorithms from \cite{cerulli2022} to the case of an inexact separation oracle, and for more general separation problems. As regards the CP algorithm with inexact oracle, we prove, under specific assumptions, a global rate of convergence for the optimality gap, and for the feasibility error. As regards the dualization approach and the IOA algorithm, we trace the steps of \cite{cerulli2022}, in a broader setting. Indeed, we consider a lower level which is not a Quadratic Programming (QP), but a Quadratically Constrained QP (QCQP) problem; we propose a restriction of this SIP problem, which is a reformulation of it when its lower level is convex; we review the sufficient condition proposed in \cite{cerulli2022}, which may be verified a posteriori on a solution of this restriction, to check if it is in fact optimal for the original SIP problem; we present the IOA algorithm with inexact oracle, and we prove that it is convergent. We refer to Table~\ref{tab:notation} in Appendix~\ref{app:notation} for a summary of notations used in the paper.

\section{Convergence rate for the CP algorithm with inexact oracle}

In the setting of convex SIP (Assumption~\ref{as:convex}) with a strongly convex objective (Assumption~\ref{as:strongconvex}), we prove a convergence rate for the Blankenship and Falk's CP algorithm \cite{blankenship1976infinitely} despite the inexactness of the separation oracle. For this purpose, we use Lagrangian duality to determine a dual of \eqref{eq:SIP} in Subsection~\ref{subsec:lagr}, and interpret this algorithm as a variant of the Frank-Wolfe algorithm applied to the dual problem of \eqref{eq:SIP} in Subsection~\ref{subsec:CP}.

\subsection{Lagrangian dual of the convex SIP problem}\label{subsec:lagr}

In line with Assumption~\ref{as:convex}, for every $y \in \mathcal{Y}$, we define $a(y) \in \mathbb{R}^m$ such that $G(x,y)= x^\top a(y)$. As $G$ is assumed to be continuous, we deduce the continuity of $a(\cdot)$. We also define the compact set $\mathcal{M} = \{  a(y)  \colon y \in \mathcal{Y} \}$ and $\mathcal{K} = \mathsf{cone}(\mathcal{M})$ the convex cone generated by $\mathcal{M}$. With this notation, program \eqref{eq:SIP} can be cast as\vspace*{-2mm}
\begin{align}
  \left\{\begin{array}{rl}
    \min\limits_{x\in \mathcal{X}} & F(x)  \\
    \text{s.t.} & x^\top z  \leq 0 \quad \forall z \in \mathcal{M}.\vspace*{-0.5em}
    \label{eq:lsip}
    \tag{\mbox{$\mathsf{SIP'}$}}
  \end{array}\right.
\end{align}
We introduce the Lagrangian function $\mathcal{L}(x,z) = F(x) + x^\top z$, defined over $\mathcal{X} \times \mathcal{K}$, and we have that 
$\mathsf{val} \text{\eqref{eq:SIP}} = \min_{x \in \mathcal{X}} \sup_{z \in \mathcal{K}} \mathcal{L}(x,z)$. 
We notice that the Lagrangian is convex w.r.t.\ $x$, and affine w.r.t.\ $z$. Since the set $\mathcal{X}$ is compact and convex (Assumption~\ref{as:convex}) and the set $\mathcal{K}$ is convex too, Sion’s minimax theorem~\cite{sion1958general} is applicable and the following holds:\vspace*{-1.5em}
\begin{align}
     \min_{x \in \mathcal{X}} \sup_{z \in \mathcal{K}} \mathcal{L}(x,z) =  \sup_{z \in \mathcal{K}} \min_{x \in \mathcal{X}}  \mathcal{L}(x,z). \label{eq:duality}
\end{align}\vspace{-0.6em}
The dual function is $\theta(z) = \min_{x \in \mathcal{X}} \mathcal{L}(x,z)$, and the dual optimization problem is\vspace{-0.2em}
\begin{align}
    \sup_{z \in \mathcal{K}} \theta(z).\vspace*{-0.5em}
    \tag{\mbox{$\mathsf{DSIP}$}}
    \label{eq:dualsip}
\end{align}
With this definition, Equation~\eqref{eq:duality} may be read as the absence of duality gap between the dual problems \eqref{eq:SIP} and \eqref{eq:dualsip}, i.e., $\mathsf{val} \text{\eqref{eq:SIP}} =  \mathsf{val} \text{\eqref{eq:dualsip}}$. 
\begin{assumption}
The function $F(x)$ is $\mu$-strongly convex. \label{as:strongconvex}
\end{assumption}
\begin{lemma}
Under Assumptions~\ref{as:convex}-\ref{as:strongconvex}, the dual function $\theta(z)$ is differentiable, with gradient $\nabla \theta(z) = \arg\min\limits_{x \in \mathcal{X}} \mathcal{L}(x,z)$. The gradient $\nabla \theta(z)$ is $\frac{1}{\mu}$-Lipschitz continuous.
\label{lem:smoothness}
\end{lemma}
\begin{proof}
    Proof in Appendix~\ref{app:smoothness}.
\end{proof}
We underline that, due to the strong convexity of $\mathcal{L}(x,z)$ with respect to $x$, its argminimum is a singleton, which is assimilated to its unique element.
\begin{lemma}
Under Assumptions~\ref{as:convex} and \ref{as:strongconvex}, for every $y,z \in \mathcal{K}$, for every $\gamma \geq 0$,\vspace{-0.8em}
\begin{align}
    \theta(z + \gamma y) \geq \theta(z) + (\nabla \theta(z)^\top y) \gamma - \frac{\lVert y \rVert^2}{2 \mu} \gamma^2.
\end{align}
\label{lem:progress}\vspace*{-1.5em}
\end{lemma}
\begin{proof}
    The proof follows the proof of \cite[Lemma~3.4]{cerulli2022}. This comes directly from the $\frac{1}{\mu}$-Lipschitzness of $\nabla \theta$.
\end{proof}%
We prove now that we can replace the sup operator with the max operator in the formulation \eqref{eq:dualsip}, under the following additional assumption.
\begin{assumption}

    There exists $\hat{x} \in \mathcal{X}$, such that $\hat{x}^\top z < 0$ for all $z \in \mathcal{M}$, i.e., $\hat{x}^\top a(y) < 0$ for all $y \in \mathcal{Y}$. \label{as:slater}
\end{assumption}
\begin{lemma}
Under Assumptions~\ref{as:convex}--\ref{as:slater}, problem \eqref{eq:dualsip} admits an optimal solution. \label{lem:dualopti}
\end{lemma}
\begin{proof}
Proof in Appendix~\ref{app:dualopti}.
\end{proof}

\subsection{The CP algorithm with inexact oracle and its dual interpretation}\label{subsec:CP} 
The CP algorithm (Alg.~\ref{alg:CP}) is a variant of Blankenship and Falk's algorithm \cite{blankenship1976infinitely} allowing for approximate solutions of the separation problems, and for more flexibility with respect to the finite set of constraints $\mathcal{M}^k$ maintained at iteration $k$. The master problem~\eqref{eq:Rk} at iteration $k$ is a relaxation of problem~\eqref{eq:lsip} since $\mathcal{M}^k \subset \mathsf{conv}(\mathcal{M})$ by construction. We compute a primal solution $x^k$ of this relaxation \eqref{eq:Rk} and the Lagrange multipliers $(\lambda^k_z)_{z\in\mathcal{M}^k}$ associated to the constraints $x^\top z \leq 0$ for $z\in\mathcal{M}^k$. The existence of these Lagrange multipliers is discussed in Remark~\ref{rem:lagrange}. 
The point $x^k$ is then provided to the $\delta$-oracle that solves the separation problem $\max\limits_{y \in \mathcal{Y}} \: G(x^k,y) = \max\limits_{y \in \mathcal{Y}} \: a(y)^\top x^k$ and returns an approximate solution $y^k$ as well as $\xi^k := a(y^k) \in \mathcal{M}$. The finite set of constraints used in the next iteration $k+1$ is defined by adding $\xi^k$ to the set $\mathcal{B}^k$, which may be either just $\mathcal{M}^k$ or a subset of its convex hull such that $z^k = \sum_{z \in \mathcal{M}^k} \lambda^k_z z$ belongs to $\mathsf{cone}(\mathcal{B}^k)$. The algorithm stops whenever $(\xi^k)^\top x^k \leq \epsilon$.
\vspace*{-1em}
\begin{algorithm}[h!]
{\small
\caption{CP algorithm for \eqref{eq:lsip}, with constraints management}
\label{alg:CP}
\begin{algorithmic}[1]
    \State{\textbf{Input:} Oracle with parameter $\delta \in [0,1)$, tolerance $\epsilon \in \mathbb{R}_+$. Let $k\gets0$, $\mathcal{M}^0 \gets \emptyset$, and
     $\nu_0 \gets \infty$.}
     \While{$\nu_k  > \epsilon$}
        \State \label{step:masterprob} Compute an optimal solution $x^k$ of the relaxation
        \begin{align}
            \left\{ \begin{array}{rl}\label{eq:Rk}
          \min\limits_{x \in \mathcal{X}} & F(x) \\
                \text{s.t.} &   x^\top z \leq 0, \quad \forall z \in \mathcal{M}^k,
    \end{array}\right.
    \tag{\mbox{$R_k$}}
    \end{align}
    and compute $z^k = \sum_{z \in \mathcal{M}^{k}} \lambda^{k}_z z$, where $\lambda^{k}_z$ is a Lagrange multiplier associated with the constraint $x^\top z \leq 0$ for $z \in \mathcal{M}^k$.
    \State Call the oracle to compute an approximate solution $y^{k} = \hat{y}(x^k)$ of $\max\limits_{y \in \mathcal{Y}} \: a(y)^\top x^k$. Define $\xi^k \gets a(y^k)  \in \mathcal{M}$.
        \State {$\mathcal{M}^{k+1} \gets \mathcal{B}^{k} \cup \{ \xi^k \} $, where $\mathcal{B}^{k}$ is any finite subset of  $\mathsf{conv}(\mathcal{M}^k)$ s.t.\ $z^k \in \mathsf{cone}(\mathcal{B}^{k})$\label{step:management}}.
        \State{$\nu_{k+1} \gets (\xi^k)^\top x^k$}
	\State {$k \gets k + 1$}
     \EndWhile
    \State  Return $x^k$.    \label{steptermCP}
\end{algorithmic}}
\end{algorithm} \vspace*{-1em}

\noindent We insist on particular cases regarding the set $\mathcal{B}^k$ used in Step~\ref{step:management} of Alg.~\ref{alg:CP}:
\begin{itemize}
    \item The set $\mathcal{B}^k$ may be equal to $\mathcal{M}^k$ at every step.
    \item The set $\mathcal{B}^k$ may be the set of atoms $z \in \mathcal{M}^k$ such that $\lambda^{k}_z > 0$, in which case Step~\ref{step:management} consists in dropping all the inactive constraints. 
    \item The set $\mathcal{B}^k$ may be a subset of $\mathsf{conv}(\mathcal{M}^k)$ of size at most $M$, and such that $z^k \in \mathsf{cone}(\mathcal{B}^k)$. For $M=2$, e.g., we can take $\mathcal{B}^k = \left\{ \frac{1}{\sum\limits_{z\in \mathcal{M}^{k}} \lambda^{k}_z} \sum\limits_{z\in \mathcal{M}^{k}} \lambda^{k}_z z \right\}$. In this case, following what is known as the constraint-aggregation approach \cite{kennedy2015}, we obtain a bounded-memory algorithm.
    \item We can also think about intermediate approaches where we do not drop every inactive constraint, but only the ones that have been staying inactive for a given number of iterations. Such strategies are also included in the framework of Alg.~\ref{alg:CP}. 
\end{itemize}

We now present a dual interpretation of Alg.~\ref{alg:CP}, as done in \cite{cerulli2022} for the case of an exact oracle. For every iteration $k$, we define the following restriction of~\eqref{eq:dualsip}:
\begin{equation}\label{eq:Dk}
\max\limits_{z \in \mathsf{cone}(\mathcal{M}^k)} \theta(z). \tag{\mbox{$D_k$}}
\end{equation}
Lemma~\ref{lem:strongdualityk} states that the restriction $\eqref{eq:Dk}$ of the problem~\eqref{eq:dualsip} may be seen as the Lagrangian dual of the relaxation $\eqref{eq:Rk}$ of \eqref{eq:lsip} solved at iteration $k$.
\begin{lemma} \label{lem:strongdualityk}
    Under Assumptions~\ref{as:convex}--\ref{as:slater}, problems $\eqref{eq:Rk}$ and $\eqref{eq:Dk}$ forms a pair of primal-dual problems and $\mathsf{val}\eqref{eq:Rk} = \mathsf{val}\eqref{eq:Dk}$. There exists a pair of optimal primal-dual solutions $(x^k, z^k)$, and, for every such pair, $x^k = \nabla \theta (z^{k})$.
\end{lemma}
\begin{proof}
    Proof in Appendix~\ref{app:strongdualityk}.
\end{proof}%
Having shown the existence of a dual optimal solution $z^k$, we can rigorously define the Lagrangian multipliers $\lambda_z^k$ used in Alg.~\ref{alg:CP} as follows.
\begin{remark} \label{rem:lagrange}
The decomposition of $z^k \in \mathsf{cone}(\mathcal{M}^k)$ as a conic combination of elements of $\mathcal{M}^k$ yields the Lagrange multipliers $(\lambda^k_z)_{z \in \mathcal{M}^k}$ invoked at Step~\ref{step:masterprob} of Alg.~\ref{alg:CP}.\end{remark}

Based on Lemma~\ref{lem:strongdualityk}, we deduce that, during the execution of Alg.~\ref{alg:CP}, the dual sequence $z^k$ instantiates the iterates of a cone-constrained fully corrective Frank–Wolfe (FCFW) algorithm \cite{locatello2017greedy} solving the dual problem \eqref{eq:dualsip}. The primal and dual interpretations of each step of the generic iteration $k$ are presented in Table~\ref{tab:CPFW}.
\begin{table}[h!]
    \centering 
    \scalebox{0.95}{
    \begin{tabular}{|c|c|c|c|c|c|}
    \hline
         & \makecell{Primal perspective:\\\textbf{CP}} & Link & \makecell{Dual perspective: \\ \textbf{FCFW}}\\
         \hline \hline
         &&&\\[-0.5em]
        \textit{Step 1} & \makecell{Solve \eqref{eq:Rk}, \\ store the solution $x^k$, \\ and the dual vector $z^k$} 
        & \makecell{Strong \\ duality} & \makecell{Solve \eqref{eq:Dk},
        \\ store the solution $z^k$, and \\ the gradient $\nabla \theta (z^{k}) = x^k$ } \\ 
        &&& \\[-0.5em]
        \hline 
        &&&\\[-0.5em]
        \textit{Step 2} & \makecell{Call the $\delta$-oracle to solve \\ $\max\limits_{y \in \mathcal{Y}} a(y)^\top x^k  ,$
  \vspace{1mm}\\ and store the solution $y^k$ } & \makecell{$\xi = a(y)$ \\ $\mathcal{M} = a(\mathcal{Y})$ \\$x^k = \nabla \theta (z^{k})$} & \makecell{Call the $\delta$-oracle to solve \vspace{1mm}\\ $\underset{\xi\in \mathcal{M}}{\max} \:  \xi^\top \nabla \theta (z^{k}),$ \vspace{1mm}\\ and store the solution $\xi^k$}\\
        &&&\\[-0.5em]
        \hline
        &&&\\[-0.5em]
         \textit{Step 3} & \makecell{$\mathcal{M}^{k+1} \gets \mathcal{B}^k \cup \{ \xi^k\}$}  &   & \makecell{$\mathcal{M}^{k+1} \gets \mathcal{B}^k \cup \{ \xi^k\}$} \\
         &&&\\[-0.5em]
         \hline
         &&&\\[-0.5em]
          \makecell{\textit{Stopping} \\ \textit{criterion}} & \makecell{$(\xi^k)^\top x^k \leq \epsilon$} & $x^k = \nabla \theta (z^{k})$ & \makecell{$ (\xi^k)^\top \nabla \theta (z^{k}) \leq \epsilon$}\\ 
         &&& \\
         \hline
    \end{tabular}}
    \centering
   \caption{The $k$-th iteration of CP, and the corresponding FCFW algorithm.}
    \label{tab:CPFW}
\end{table}

\subsection{Convergence rate for the CP algorithm with inexact oracle}

We define the constant $R= \sup_{z \in \mathcal{M}} \lVert z \rVert = \sup_{z \in \mathsf{conv}(\mathcal{M})} \lVert z \rVert $. According to Lemma~\ref{lem:dualopti}, problem~\eqref{eq:dualsip} admits an optimal solution $z^* \in \mathcal{K}$. We define $\tau = \inf \{ t \geq 0 \colon z^* \in t \: \mathsf{conv}(\mathcal{M}) \}$. The scalar $\tau$ plays a central role in the convergence rate analysis of the CP algorithm with inexact oracle, conducted in the following theorem. An interesting future research direction is finding an efficient approach to estimate (an upper bound on) $\tau$  without computing $z^*$.
\begin{theorem}
    Under Assumption~\ref{as:convex}-\ref{as:slater}, denoting by $x^*$ an optimal solution of \eqref{eq:lsip}, if Alg.~\ref{alg:CP} executes iteration $k \in \mathbb{N}$, then \vspace{-0.8em}
    \begin{align}
        F(x^*) - F(x^k) \leq \frac{2 \: R^2 \tau^2}{\mu \: (1-\delta)^2} \: \frac{1}{k+2}. \vspace*{-0.5em}\label{eq:optimality_gap} 
    \end{align} \label{th:convrate}\vspace{-1em}
\end{theorem}
\begin{proof}
We define the optimality gap $\Delta_k = F(x^*) - F(x^k)  = \mathsf{val}\text{\eqref{eq:lsip}} - F(x^k)$. We emphasize that at each iteration $k$, $ \theta(z^{k}) = F(x^k)$, thus $\Delta_k$ may also be seen as the optimality gap in the dual problem~\eqref{eq:dualsip}, i.e., $\Delta_k = \mathsf{val}\text{\eqref{eq:dualsip}} - F(x^k) = \theta(z^*) - \theta(z^{k})$. We prove the inequality~\eqref{eq:optimality_gap} by induction. When $\tau = 0$ (i.e., 0 is a dual optimal solution), inequality~\eqref{eq:optimality_gap} holds trivially. 
Indeed, for every $k$, $0$ is a feasible point of the dual problem of \eqref{eq:Rk}, thus $\theta(0) \leq \theta(z^k)$. Since $\theta(z) = \min_{x \in \mathcal{X}} F(x) + x^\top z $,  $\theta(0) = \min_{x \in \mathcal{X}} F(x) = F(x^*).$ Combining this with the fact that, for every $k,$ $F(x^k) \leq F(x^*)$ and $\theta(z^k) = F(x^k)$, we have: $F(x^k) \leq F(x^*) = \theta(0) \leq \theta(z^k) = F(x^k)$. Thus, $F(x^k)  = F(x^*)$, and inequality~\eqref{eq:optimality_gap} holds.
We assume now that $\tau > 0$.
\vspace*{-0.6em}

\paragraph{\textbf{Base case (}$k=0$\textbf{).}} Since $\theta$ is concave, $\Delta_0 =  \theta(z^*) - \theta(z^0)  \leq \nabla \theta(z^0)^\top  (z^* - z^0) = \theta(z^0)^\top  z^* ,$
with the last equality following from $z^0 = 0$ (as $\mathcal{M}^0 = \emptyset)$. We remark that $\nabla \theta(z^0)^\top  z^*  =  (\nabla\theta(z^0)-\nabla\theta(z^*))^\top  z^*$ since $\nabla\theta(z^*)^\top  z^* = 0$ by optimality of $z^*$. Hence, $\Delta_0 \leq (\nabla \theta(z^0) -  \nabla \theta(z^*))^\top  z^*  \leq \lVert \nabla \theta(z^0) -  \nabla \theta(z^*) \rVert \: \lVert z^* \rVert$, where the last inequality is the Cauchy-Schwartz inequality. Using the $\frac{1}{\mu}$-Lipschitzness of $\nabla \theta$ (Lemma~\ref{lem:smoothness}), we know that $\lVert \nabla \theta(z^0) -  \nabla \theta(z^*) \rVert \leq \frac{1}{\mu} \lVert z^0 - z^* \rVert  = \frac{1}{\mu} \lVert z^* \rVert$. Since $z^* \in \tau \mathsf{conv}(\mathcal{M})$,
$\Delta_0 \leq \frac{1}{\mu} \lVert z^* \rVert^2 \leq \frac{(R\tau)^2}{\mu} \leq \frac{(R\tau)^2}{(1-\delta)^2\mu}$ as  $1-\delta \in (0,1]$.
\vspace*{-0.6em}

\paragraph{\textbf{Induction.}} We suppose that the algorithm runs $k+1$ iterations and does not meet the stopping condition; we assume that property~\eqref{eq:optimality_gap} is true for $k$. Since $z^k \in \mathsf{cone}(\mathcal{B}^k)$, and $\mathcal{M}^{k+1} = \mathcal{B}^k \cup \{ \xi^k \}$, we deduce that $z^k + \gamma \xi^k \in \mathsf{cone}(\mathcal{M}^{k+1})$, for every $\gamma \geq 0$, implying $\theta(z^{k+1}) \geq \theta(z^k + \gamma \xi^k)$. Moreover, Lemma~\ref{lem:progress} yields a lower bound on the progress made during iteration $k+1$:\vspace{-1em}
    \begin{align}
        \theta(z^{k+1}) \geq \theta(z^k + \gamma \xi^k) \geq \theta(z^k) + \gamma \: \nabla \theta(z^k)^\top  \xi^k  - \frac{\lVert  \xi^k \rVert^2}{2 \mu} \gamma^2,
    \end{align}
for every $\gamma \geq 0$. Multiplying by $-1$, adding $\theta(z^*)$ to both left and right-hand sides of the above inequality, and using $\lVert \xi^{k} \rVert \leq R$, we have that  
    \begin{equation}
        \Delta_{k+1} \leq \Delta_{k} - \gamma \: \nabla \theta (z^{k})^\top \xi^{k}   + \frac{R^2}{2\mu} \gamma^2,
         \label{eq:gamma}
    \end{equation}
for every $\gamma \geq 0$. In addition, by concavity of $\theta,$ $\Delta_k = \theta(z^*) - \theta(z^{k}) \leq  \nabla \theta (z^{k})^\top (z^* - z^{k})$. Note that we have $\nabla \theta (z^k)^\top z^k = 0$, following from the first-order optimality condition holding at $1$ of the differentiable function $\alpha(t) = \theta(t z^k)$. Indeed, $\alpha'(1) =(\nabla \theta (z^k))^\top z^k = 0$, because (i) $1$ is optimal for $\alpha$ since $z^k \in \underset{z \in \mathsf{cone}(\mathcal{M}^k)}{\text{argmax}} \theta(z)$, (ii) $1$ lies in the interior of the definition domain of $\alpha$.

Thus, $ \Delta_k \leq \nabla \theta (z^k)^\top z^*$. As $z^* \in \tau\, \mathsf{conv}(\mathcal{M})$, 
    \begin{equation}
       \Delta_k \leq \max_{z \in \tau \mathsf{conv}(\mathcal{M})} \nabla \theta (z^k)^\top z = \tau \: \max_{z \in  \mathcal{M}} \nabla \theta (z^k)^\top z = \tau  \phi(x^k),
       \label{eq:control}
    \end{equation} 
where the last equality follows from $\nabla \theta(z^k) = x^{k}$, and from the definition of the value function $\phi(x) = \max_{y \in \mathcal{Y}} x^\top a(y)$. The stopping criterion $a(y^k)^\top x^k = (\xi^k)^\top x^k \leq \epsilon$ is not met at the end of iteration $k$, as iteration $k+1$ is executed. Therefore, $\phi(x^{k}) \geq a(y^k)^\top x^k > \epsilon \geq 0$. Inequality~\eqref{eq:oraclebound} yields
$\phi(x^k)- G(x^k,y^k) \leq \delta \phi(x^k)$, i.e., $(1-\delta)\phi(x^k)  \leq  G(x^k,y^k) = (x^k)^\top a(y^k) = \nabla \theta(z^k)^\top \xi^k$. Therefore, as we have $\tau > 0$, we deduce from Eq.~\eqref{eq:control} that \vspace*{-0.4em}
\begin{align}
     \frac{1-\delta}{\tau}\Delta_k \leq  \nabla \theta(z^k)^\top \xi^k.
    \label{eq:control2}
\end{align}     \vspace*{-5mm}

Combining Eqs.~\eqref{eq:gamma} and \eqref{eq:control2}, we obtain $
 \Delta_{k+1} \leq \Delta_{k} - \gamma \frac{1-\delta}{\tau} \Delta_{k}  + \frac{R^2}{2 \mu} \gamma^2,$
for every $\gamma \geq 0$. Factoring and setting $\tilde{\gamma} = \gamma \frac{1-\delta}{\tau}$ (for every $\tilde{\gamma} \geq 0$) yields:\vspace{-0.8em}
\begin{equation}
     \Delta_{k+1} \leq (1 - \tilde{\gamma}) \Delta_{k}   + \frac{R^2 \tau^2}{2 \mu (1-\delta)^2} \tilde{\gamma}^2.
     \label{eq:progress}
\end{equation}
 Applying Eq.~\eqref{eq:progress} with $\tilde{\gamma} = \frac{2}{k+2}$, and defining $C = \frac{ 2 R^2 \tau^2}{ \mu (1-\delta)^2}$ we obtain:
{\small   \begin{align*}
     \Delta_{k+1} \leq (1 - \frac{2}{k+2}) \Delta_{k}   + \frac{C}{(k+2)^2}
     \leq \frac{k}{k+2} \frac{C}{k+2} + \frac{C}{(k+2)^2},
\end{align*}}
with the second inequality coming from the application of \eqref{eq:optimality_gap}, which holds for $k$ by the induction hypothesis. Finally, we deduce that
{\small  \begin{align*}
     \Delta_{k+1} \leq \frac{C}{k+2} (\frac{k}{k+2}  + \frac{1}{k+2}) \leq \frac{C}{k+2} \frac{k+1}{k+2} \leq \frac{C}{k+2} \frac{k+2}{k+3} = \frac{C}{k+3},
\end{align*}}
where the third inequality follows from the observation that $\frac{k+1}{k+2} \leq \frac{k+2}{k+3}$. Hence, the property \eqref{eq:optimality_gap} is true for $k+1$ as well. This concludes the proof. \end{proof}

The main difference between Eq.~\eqref{eq:optimality_gap} and the convergence rate of the CP algorithm with exact oracle considered in \cite{cerulli2022} is exactly the term $\frac{1}{(1-\delta)^2}$, which is related to the inexactness of the oracle here considered. 
Furthermore, the convergence result of Theorem~\ref{th:convrate} differs from \cite[Th.~2]{locatello2017greedy}, where the existence of a scalar $t \geq 0$ such that $\{z^*\} \cup \{ z^k\}_{k \in\mathbb{N}} \subset t\: \mathsf{conv}(\mathcal{M})$ is assumed. Indeed, the present analysis only uses the property $z^* \in \tau \: \mathsf{conv}(\mathcal{M})$, regardless of whether the iterates $(z^k)_{k \in \mathbb{N}}$ also belong to this set or not. Therefore the objective error estimate in Eq.~\eqref{eq:optimality_gap} is independent of the particular sequence $(z^k)_{k \in \mathbb{N}}$ resulting from the choice of the sets $\mathcal{B}^k$ at Step~\ref{step:management} of Alg.~\ref{alg:CP} (constraint management strategy).

The following theorem states that smallest feasibility error among the $k$ first iterates $x_1, \dots, x_k$ follows a $O(\frac{1}{k})$ convergence rate.
\begin{theorem}\label{th:CP}
    Under Assumptions~\ref{as:convex}--\ref{as:slater}, if Alg.~\ref{alg:CP} executes iteration $k$, for $k \geq 2$, then  \vspace*{-3mm}
    \begin{align}\label{eq:th5}
       \min_{0\leq \ell \leq  k } \phi(x^\ell) \leq \frac{27 \:  R^2 \tau}{4 \mu (1-\delta)^2} \: \frac{1}{k+2}.
    \end{align}
\end{theorem}

\begin{proof}
If $\tau =0,$ inequality~\eqref{eq:th5} holds trivially because, from Theorem~\ref{th:convrate}, $F(x^k) = \mathsf{val}\text{\eqref{eq:lsip}}$, thus $\phi(x^k) \leq 0.$ We consider the case $\tau >0$.
We keep the definition of the constant $C = \frac{ 2 R^2 \tau^2}{ \mu \: (1-\delta)^2}$, and we define the constants $\alpha = \frac{2}{3}$, $\beta = \frac{27}{8\tau}$, and $D = k +2$. Let us suppose that \vspace*{-1.2em}
\begin{align}
  \phi(x^\ell) > \frac{\beta C}{D}, \forall  \ell \in \{0, \dots k\}.
  \label{eq:assumptioninfeas}
\end{align}
We will show a contradiction. If iteration $k+1$ is not executed because the algorithm stopped at iteration $k$, we still define $z^{k+1} = \underset{z = z^k + \gamma \xi^k , \gamma \geq 0 }{\text{argmax}} \theta(z)$, and $\Delta_{k+1} = \theta(z^*) - \theta(z^{k+1}) \geq 0$. Hence, regardless whether the iteration $k+1$ is executed or not, $\Delta_{\ell}$ and $\Delta_{\ell+1}$ are well defined for all $\ell \in \lbrace 0, \dots, k \rbrace$, and  we can apply Eq.~\eqref{eq:gamma} to deduce  $\Delta_{\ell+1} \leq \Delta_{\ell} - \gamma \: \nabla \theta (z^{\ell})^\top \xi^{\ell}   + \frac{R^2}{2\mu} \gamma^2$, for every $\gamma \geq 0$. Eq.~\eqref{eq:assumptioninfeas} implies that $\phi(x^\ell) >0$, and from Eq.~\eqref{eq:oraclebound}, we deduce that  $\phi(x^\ell)- G(x^\ell,y^\ell) \leq \delta \phi(x^\ell)$, i.e., $(1-\delta)\phi(x^\ell)  \leq  G(x^\ell,y^\ell) = (x^\ell)^\top a(y^\ell) = \nabla \theta(z^\ell)^\top \xi^\ell$. Combining this with Eq.~\eqref{eq:assumptioninfeas}, we deduce that $\nabla \theta(z^\ell)^\top \xi^\ell > \frac{(1 -\delta)\beta C}{D}$, and therefore, for every $\gamma  >  0$, $\Delta_{\ell+1} < \Delta_{\ell} - \gamma \: \frac{(1 -\delta)\beta C}{D}   + \frac{R^2}{2\mu} \gamma^2$. Applying this inequality for $\gamma = \frac{2\tau}{(1-\delta)(\ell+2)} > 0$, we obtain
{\small\begin{align}
    \Delta_{\ell+1} < \Delta_{\ell} - & \frac{2 \tau \beta C}{(\ell +2) D}   + \frac{2 R^2 \tau^2}{\mu (1-\delta)^2} \frac{1}{(\ell+2)^2} = \Delta_{\ell} - \frac{2 \tau \beta C}{(\ell +2) D}   +  \frac{C}{(\ell+2)^2}. \label{eq:ineqC}
\end{align}}
We define $k_{\min} = \lceil \alpha D \rceil - 2$, and we notice that $k_{\min} \geq 0$, since $D\geq 4$. Furthermore, for every $\ell \in \{ k_{\min}, \dots k \}$, $\alpha D \leq \ell + 2 \leq D$. Combining this with Eq.~\eqref{eq:ineqC}, we know that, for every $\ell \in \{ k_{\min}, \dots k \}$,\vspace*{-0.2em}
{\small\begin{align}
    \Delta_{\ell+1} < \Delta_{\ell} - \frac{2 \tau \beta C}{D^2}   +  \frac{C}{\alpha^2 D^2} = \Delta_{\ell} + \frac{C}{D^2}(\frac{1}{\alpha^2}- 2\tau \beta).
\end{align}}
Summing these inequalities for $\ell  \in \{ k_{\min}, \dots k \}$, we obtain 
{\small\begin{align}
    \Delta_{k+1} < \Delta_{k_{\min}} +  \frac{C (k + 1 - k_{\min} )}{D^2}(\frac{1}{\alpha^2} - 2\tau \beta),
\end{align}}
and, using the bound on the objective gap at iteration $k_{\min}$ given by Theorem~\ref{th:convrate}, we have $\Delta_{k+1} < \frac{C}{k_{\min} + 2} + \frac{C (k +1 - k_{\min} )}{D^2}(\frac{1}{\alpha^2} - 2\tau \beta )$.
We notice that $k_{\min} + 2 \geq \alpha D$, and $ k + 1 - k_{\min} \geq (1-\alpha) D$. By definition of $\alpha$ and $\beta$, $\frac{1}{\alpha^2} - 2\tau \beta = \frac{9}{4}- \frac{27}{4} \leq 0$, and thus
{\small\begin{align}
    \Delta_{k+1} < \frac{C}{\alpha D} + \frac{C (1-\alpha)}{D}(\frac{1}{\alpha^2} - 2\tau \beta ) = \frac{C}{\alpha D}( 1 + \frac{1-\alpha}{\alpha} - 2\alpha(1-\alpha) \tau \beta).
\end{align}}
Using again that $\alpha = \frac{2}{3}$, we deduce that $\Delta_{k+1} < \frac{C}{\alpha D}(\frac{3}{2} - \frac{4}{9} \tau \beta)$. Since $\beta = \frac{27}{8\tau}$, we have $(\frac{3}{2} - \frac{4}{9} \tau \beta) = 0$. Therefore, we obtain $\Delta_{k+1}  < 0$, which contradicts the definition of $\Delta_{k+1}$. We can conclude that the assumption at Eq.~\eqref{eq:assumptioninfeas} cannot hold, and there exists $\ell \in \{0, \dots, k\}$ such that $\phi(x^\ell) \leq \frac{\beta C}{D} = \frac{27 \:  R^2 \tau}{4 \mu (1-\delta)^2} \: \frac{1}{k+2}$.
\end{proof}

    The proof of Theorem~\ref{th:CP} is inspired by a previous work on the FCFW algorithm \cite{jaggi2013revisiting}, with some adaptations to the problem~\eqref{eq:dualsip}, a convex optimization problem on a cone. Indeed, \cite[Th.~2]{jaggi2013revisiting} considers the duality gap convergence for the FCFW algorithm in the case of an {optimization problem over a} convex and compact {domain} for which a linear minimization oracle \cite{jaggi2013revisiting} is available. Yet, in the case of problem~\eqref{eq:dualsip}, the optimization set $\mathsf{cone}(\mathcal{M})$ is not compact, and we do not assume, as done in \cite{locatello2017greedy}, that the iterates $(z^k)_{k\in \mathbb{N}}$ belong to a same compact set of the form $t \: \mathsf{conv}(\mathcal{M})$. 
    Therefore, we propose an adapted result and proof for such a case, which only rely on the property $z^* \in \tau \: \mathsf{conv}(\mathcal{M})$, regardless of whether the iterates $(z^k)_{k\in \mathbb{N}}$ also belong to this set.

\section{IOA algorithm with inexact oracle}

The CP algorithm provides iterates that are not feasible for the problem~\eqref{eq:SIP} until the algorithm converges; feasibility is obtained only asymptotically. To overcome this limitation, the authors of \cite{cerulli2022} proposed an IOA algorithm that generates a minimizing sequence of points that are feasible in \eqref{eq:SIP}, in the case where the inner problem is a QP problem, solved through an exact separation oracle. We extend this algorithm to the case where the inner problem is a QCQP problem, and the separation oracle is inexact. In particular, we consider the setting defined in the following assumption.

\begin{assumption} The parameterization of the SIP constraints satisfies the following:
    \begin{itemize}
    \item  Linear mappings $x \mapsto Q(x) \in \mathbb{S}_n$, $x \mapsto q(x) \in \mathbb{R}^n$ and $x \mapsto b(x) \in \mathbb{R}$ exist such that $$G(x,y) =  - \frac{1}{2} y^\top Q(x) y + q(x)^\top y + b(x).$$
    \item There exist $Q^1, \dots, Q^r \in \mathbb{S}_n$, $q^1, \dots, q^r \in \mathbb{R}^n$ and $b_1, \dots, b_r \in \mathbb{R}$ such that $$\mathcal{Y} = \left \lbrace y \in \mathbb{R}^n \colon \frac{1}{2} y^\top Q^j y +  (q^j) ^\top y + b_j \leq 0, \forall j \in \{1, \dots, r \} \right \rbrace.$$ 
    \vspace*{-5mm}
    
    \noindent We also assume to know $\rho \geq 0$ such that $\mathcal{Y} \subset B(0, \rho)$.
   \end{itemize}   \label{as:subpbqcqp}
\end{assumption}

In the context of Assumption~\ref{as:subpbqcqp}, a possible way to deal with the SIP problem \eqref{eq:SIP} is what is called \textit{lower-level dualization approach} in \cite{cerulli2022}, which consists in replacing the constraint involving the QCQP inner problem with one involving its dual. In particular, we consider a strong dual of a Semidefinite Programming (SDP) relaxation of the inner problem (or a reformulation if the latter is convex). In Subsection~\ref{subsec:sdprelax}, we introduce the classical SDP relaxation of the inner problem (reformulation, if it is convex) regularized by a ball constraint, and then, we introduce the SDP dual of this relaxation (reformulation, resp.). In Subsection~\ref{subsec:restr_reform} we present the finite formulation~\eqref{eq:SIP:convexreform}, obtained by applying the \textit{lower-level dualization approach} to the problem~\eqref{eq:SIP}. This formulation is a reformulation of \eqref{eq:SIP} if $Q^1, \dots, Q^r$ are Positive Semidefinite (PSD) and $Q(x)$ is PSD for every $x \in \mathcal{X}$. Otherwise, an a posteriori sufficient condition on a computed optimal solution $\bar{x}$ of \eqref{eq:SIP:convexreform} introduced in Subsection~\ref{sec:sufficientcondition} can be verified. If $\bar{x}$ satisfies such a condition, one can state that it is an optimal solution of \eqref{eq:SIP}. If not, the IOA algorithm is proposed in Subsection~\ref{sec:IO}, which generates a sequence of converging feasible points of \eqref{eq:SIP:convexreform}.
\subsection{SDP relaxation/reformulation of the inner problem} \label{subsec:sdprelax}

In this section, we reason for any fixed value of the decision vector $x \in \mathcal{X}$. The corresponding inner problem $\max_{y \in \mathcal{Y}} G(x,y)$ is the following QCQP problem

\begin{equation}\label{eq:QCQP}
    \left\{ \begin{array}{cll}
         \max \limits_{y \in \mathbb{R}^n} &  -\frac{1}{2} y^\top Q(x) y + q(x)^\top y +b(x) & \\
         \text{s.t.} & \frac{1}{2} y^\top Q^j y +  (q^j) ^\top y + b_j \leq 0 & \forall j \in \{ 1, \dots, r \}.
    \end{array}     \tag{\mbox{$\mathsf{P}_x$}} \right.
\end{equation} 
We define the linear matrix operator $\mathbf{\mathcal{Q}}(x) = \frac{1}{2} \begin{pmatrix} -Q(x) & q(x) \\ q(x)^\top & 2 b(x) \end{pmatrix} \in \mathbb{S}_{n+1}$, and the matrices $\mathbf{\mathcal{Q}}^j = \frac{1}{2} \begin{pmatrix} Q^j & q^j \\ (q^j)^\top & 2 b_j \end{pmatrix} \in \mathbb{S}_{n+1}$, for $j \in \{1, \dots, r\}$, the identity matrix $I_{n+1} \in  \mathbb{S}_{n+1}$, as well as $E := \mathsf{diag}(0, \dots, 0,1) \in \mathbb{S}_{n+1}$. 
With this notation, we introduce the following SDP problem
\begin{equation}\label{eq:SDP_relax}
  \left\{  \begin{array}{cll}
         \max \limits_{Y \in \mathbb{S}_{n+1}} & \langle  \mathbf{\mathcal{Q}}(x), Y \rangle  & \\
         \text{s.t.} & \langle \mathbf{\mathcal{Q}}^j, Y \rangle  \leq  0 &  \forall j \in \{ 1, \dots, r \} \\
         & \langle I_{n+1}, Y \rangle \leq  1+\rho^2 & \\
         & \langle E, Y \rangle = 1 & \\
         & Y \succeq  0. &
    \end{array} \tag{\mbox{$\mathsf{SDP}_x$}} \right.
\end{equation}
Under Assumption~\ref{as:subpbqcqp}, this problem is the standard Shor SDP relaxation of \eqref{eq:QCQP}. This convex relaxation is obtained by first reformulating~\eqref{eq:QCQP} in a lifted space with an additional rank constraint and then dropping such a constraint. If the inner problem~\eqref{eq:QCQP} is convex, i.e., $Q^1, \dots, Q^r$ are PSD and $Q(x)$ is PSD for every $x \in \mathcal{X}$, both problems have the same optimal objective value. These results are proven in the following lemma.
\begin{lemma}\label{lem:qcqpsdp}
    Under Assumption~\ref{as:subpbqcqp}, $\mathsf{val} \text{\eqref{eq:SDP_relax}} \geq \mathsf{val} \text{\eqref{eq:QCQP}}$. If, moreover, $Q(x), Q^1, \dots, Q^r$ are PSD, then $\mathsf{val} \text{\eqref{eq:SDP_relax}}=\mathsf{val} \text{\eqref{eq:QCQP}}$. \label{lem:qcqpsdp}
\end{lemma}
\begin{proof}
    Proof in Appendix~\ref{app:qcqpsdp}.
\end{proof}

The following SDP problem
\begin{equation}
   \left\{ \begin{array}{cl}
         \min \limits_{\lambda, \alpha, \beta} &  \alpha (1+\rho^2) + \beta\\
         \text{s.t.} &  \sum\limits_{j=1}^r  \lambda_j \mathbf{\mathcal{Q}}^j  + \alpha I_{n+1} + \beta E \succeq \mathbf{\mathcal{Q}}(x) \\
         & \lambda \in \mathbb{R}^r_+, \: \alpha \in \mathbb{R}_+, \beta \in \mathbb{R},
    \end{array}
    \tag{\mbox{$\mathsf{DSDP}_x$}} \right.
    \label{eq:DSDP}
\end{equation}
is the dual of problem \eqref{eq:SDP_relax}, as the following lemma states. 
\begin{lemma}\label{lem:strong_duality}
Formulations \eqref{eq:SDP_relax} and \eqref{eq:DSDP} are a primal-dual pair of SDP problems and strong duality holds, thus $\mathsf{val}\text{\eqref{eq:SDP_relax}} = \mathsf{val}\text{\eqref{eq:DSDP}}.$
\end{lemma}
\begin{proof}
Proof in Appendix~\ref{app:strong_duality}.
\end{proof}

\subsection{SDP restriction/reformulation of the SIP problem}\label{subsec:restr_reform}

Leveraging Section~\ref{subsec:sdprelax}, which focuses on the inner problem~\eqref{eq:QCQP}, its SDP relaxation~\eqref{eq:SDP_relax} and the respective dual problem~\eqref{eq:DSDP}, we propose a single-level finite restriction of problem~\eqref{eq:SIP}. It is a reformulation of \eqref{eq:SIP} if $Q^1, \dots, Q^r$ are PSD, and $Q(x)$ is PSD for every $x \in \mathcal{X}$.

\begin{theorem}\label{th:finiteformulation}
Under Assumption~\ref{as:subpbqcqp}, the finite formulation
{\small\begin{align}
\left\{ \begin{array}{rl}
  \min\limits_{x,\lambda,\alpha,\beta} & F(x)   \\
  \text{s.t.} &   \alpha (1 + \rho^2) + \beta \leq 0  \\
  & \sum\limits_{j=1}^r  \lambda_j \mathbf{\mathcal{Q}}^j  + \alpha I_{n+1} + \beta E - \mathbf{\mathcal{Q}}(x) \succeq 0\\
& x \in \mathcal{X}, \lambda \in \mathbb{R}^r_+, \; \alpha \in \mathbb{R}_+, \beta \in \mathbb{R},
  \label{eq:SIP:convexreform}
  \tag{\mbox{$\mathsf{SIPR}$}}
  \end{array}\right.
  \end{align}} 
is a restriction of problem \eqref{eq:SIP}. If $Q^1, \dots, Q^r$ are PSD, and if $Q(x)$ is PSD for every $x \in \mathcal{X}$, then the finite formulation \eqref{eq:SIP:convexreform} is a reformulation of \eqref{eq:SIP}.
\end{theorem}
\begin{proof}
Let $\mathsf{Feas}$\eqref{eq:SIP} and $\mathsf{Feas}$\eqref{eq:SIP:convexreform} be the feasible sets of \eqref{eq:SIP} and \eqref{eq:SIP:convexreform}, respectively. Since \eqref{eq:SIP} and \eqref{eq:SIP:convexreform} share the same objective function, proving for every $x\in \mathcal{X}$ the implication \vspace*{-0.6em}
 \begin{equation}\label{eq:feas}
     \left( \exists \; \lambda \in \mathbb{R}^r_+, \; \alpha \in \mathbb{R}_+,\; \beta \in \mathbb{R} :\; (x,\lambda,\alpha,\beta) \in  \mathsf{Feas}\text{\eqref{eq:SIP:convexreform}} \right)\; \Longrightarrow \; x \in \mathsf{Feas}\text{\eqref{eq:SIP}},
 \end{equation} 
will prove the first part of the theorem. For every $x\in \mathcal{X}$, we have:\vspace*{-0.8em}
\begin{equation}
    \label{eq:implication}
          \mathsf{val}\text{\eqref{eq:SDP_relax}}   \leq 0
     \Longrightarrow 
         \mathsf{val}\text{\eqref{eq:QCQP}} \leq 0
     \iff x \in \mathsf{Feas}\text{\eqref{eq:SIP}},
\end{equation}
where the implication stems from the fact that $\mathsf{val}$\eqref{eq:QCQP} $\leq \mathsf{val}$\eqref{eq:SDP_relax}. Applying the strong duality Lemma~\ref{lem:strong_duality}, we obtain that, for every $x \in \mathcal{X}$, 
{\small\begin{align}\label{eq:impl2}
    \hspace*{-2.9mm}\mathsf{val}\text{\eqref{eq:SDP_relax}} \leq 0 \hspace{-1mm}& \hspace{-0.5mm} \iff \mathsf{val}\text{\eqref{eq:DSDP}} \leq 0 \\
     \hspace{-1mm}& \hspace{-0.5mm} \iff 
	 \exists \, \lambda \in \mathbb{R}^r_+, \, \alpha \in \mathbb{R}_+,\, \beta \in \mathbb{R} : \left\{\begin{array}{l} \alpha (1 + \rho^2) + \beta \leq 0 \\ \sum\limits_{j=1}^r  \lambda_j \mathbf{\mathcal{Q}}^j  + \alpha I_{n+1} + \beta E - \mathbf{\mathcal{Q}}(x) \succeq 0 \: . \end{array} \right.  \\
     \hspace{-1mm}& \hspace{-0.5mm}\iff \exists\, \lambda \in \mathbb{R}^r_+, \, \alpha \in \mathbb{R}_+,\, \beta \in \mathbb{R},\; (x,\lambda,\alpha,\beta) \in  \mathsf{Feas}\text{\eqref{eq:SIP:convexreform}}. \label{eq:impl3}
\end{align}}

The equivalence \eqref{eq:impl3}, together with implication \eqref{eq:implication}, prove the implication \eqref{eq:feas}.

If $Q^1, \dots, Q^r$ are PSD, and if $Q(x)$ is PSD for every $x \in \mathcal{X}$, we can replace the implication \eqref{eq:implication} by the equivalence
$ \mathsf{val}\text{\eqref{eq:SDP_relax}} \leq 0  \iff           \mathsf{val}\text{\eqref{eq:QCQP}} \leq 0   \iff x \in \mathsf{Feas}$\eqref{eq:SIP}. 
This, together with equivalence \eqref{eq:impl3}, proves that\vspace{-0.5em}
\begin{equation*}
\exists \; \lambda \in \mathbb{R}^r_+, \; \alpha \in \mathbb{R}_+,\; \beta \in \mathbb{R} :\; (x,\lambda,\alpha,\beta) \in  \mathsf{Feas}\text{\eqref{eq:SIP:convexreform}} \; \iff \; x \in \mathsf{Feas}\text{\eqref{eq:SIP}},\vspace{-0.5em}
\end{equation*} 
i.e., \eqref{eq:SIP:convexreform} is a reformulation of \eqref{eq:SIP}, having the same objective function.
\end{proof}

Note that under Assumptions~\ref{as:convex} and \ref{as:subpbqcqp}, the finite formulation \eqref{eq:SIP:convexreform} is convex. 

\subsection{Optimality of the restriction: a sufficient condition}\label{sec:sufficientcondition}

Theorem~\ref{th:finiteformulation} states that the single-level finite formulation \eqref{eq:SIP:convexreform} is an exact reformulation of the problem \eqref{eq:SIP}, if $Q^1,\dots Q^r$ are PSD, and $Q(x) \succeq 0$ for all $x \in \mathcal{X}$. Even if this \textit{a priori} condition is not satisfied for all $x\in \mathcal{X}$, as what is done in \cite{cerulli2022} in a different setting, an \textit{a posteriori} condition on the computed solution $\bar{x}$ of \eqref{eq:SIP:convexreform} enables us to state that $\bar{x}$ is an optimal solution of \eqref{eq:SIP}. Assumptions~\ref{as:convex} and \ref{as:subpbqcqp} are fundamental here to ensure that the feasible sets of \eqref{eq:SIP} and \eqref{eq:SIP:convexreform} are closed and convex.

\begin{theorem}
\label{th:restriction_optimality}
Under Assumptions~\ref{as:convex} and \ref{as:subpbqcqp}, and assuming $Q^1, \dots, Q^r$ are PSD, let $\bar{x}$ be a solution of the single-level formulation \eqref{eq:SIP:convexreform}. If $Q(\bar{x}) \succ 0$, then $\bar{x}$ is optimal in \eqref{eq:SIP}.
\end{theorem}
\begin{proof}
Given a closed convex set $S$, according to Def.~5.1.1 in \cite[Chap.~III]{hiriart2013convex}, the tangent cone to $S$ at $x$ (denoted by $T_S(x)$) is the set of directions $u \in \mathbb{R}^m$ such that there exist a sequence $(x_k)_{k \in \mathbb{N}}$ in $S$, and a positive sequence $(t_k)_{k \in \mathbb{N}}$ s.t.\ $t_k \rightarrow 0$ and $\frac{x_k - x}{t_k} \rightarrow u$. Moreover, according to Def.~5.2.4 in \cite[Chap.~III]{hiriart2013convex}, the normal cone $N_S(x)$ to $S$ at $x$ is the polar cone of the tangent cone $T_S(x)$, i.e., $N_S(x) = T_S(x)^\circ$. We define the closed convex set $C$ (resp.\ $\hat{C}$) as the feasible set of formulation \eqref{eq:SIP} (resp.\ \eqref{eq:SIP:convexreform}).

Since $Q(\bar{x}) \succ 0$, the set of positive definite matrices is open, and $Q(x)$ is continuous, there exists $r >0$ s.t.\ for all $x$ in the open ball of radius $r$ with center $\bar{x}$ (denoted by $B(\bar{x},r)$) $Q(x) \succeq 0$. This means that for all $x$ in $\mathcal{X} \cap B(\bar{x},r)$, $\mathsf{val}\text{\eqref{eq:QCQP}} = \mathsf{val}$\eqref{eq:SDP_relax}. Hence, we deduce that, for every $x \in \mathcal{X} \cap B(\bar{x},r)$, $x$ is feasible in \eqref{eq:SIP} if and only if $x$ is feasible in \eqref{eq:SIP:convexreform}. In other words, $C \cap B(\bar{x},r) = \hat{C} \cap B(\bar{x},r)$. According to the aforementioned definition of the tangent and normal cones, we further deduce that $T_C(\bar{x}) = T_{\hat{C}}(\bar{x})$, and $N_C(\bar{x})= T_C(\bar{x})^\circ =T_{\hat{C}}(\bar{x})^\circ =  N_{\hat{C}}(\bar{x})$.

We know that $\bar{x}$ is optimal in \eqref{eq:SIP:convexreform}, i.e., $\bar{x} \in \arg\min_{x \in \hat{C}} F(x)$. Since $F$ is a finite-valued convex function, and $\hat{C}$ is a closed and convex set, the assumptions of Theorem~1.1.1 in \cite[Chap.~VII]{hiriart2013convex} hold, and we can deduce that $0 \in \partial F(\bar{x}) +  N_{\hat{C}}(\bar{x})$. Using the equality $N_C(\bar{x}) = N_{\hat{C}}(\bar{x})$, we have that $0 \in \partial F(\bar{x}) +  N_{C}(\bar{x})$ too. Applying the same theorem with the closed and convex set $C$, we know that $0 \in \partial F(\bar{x}) +  N_{C}(\bar{x})$ implies that $\bar{x} \in \arg\min\limits_{x \in C} F(x)$, meaning that $\bar{x}$ is optimal in \eqref{eq:SIP}.
\end{proof}

\subsection{The IOA algorithm}\label{sec:IO}
If neither the lower level is convex, nor the sufficient optimality condition in Theorem~\ref{th:restriction_optimality} is satisfied, we do not directly obtain an optimal solution of problem~\eqref{eq:SIP} by solving the finite formulation \eqref{eq:SIP:convexreform}. In this section, we present an algorithm based on the lower-level dualization approach (presented above) and on an inexact separation oracle, that generates a minimizing sequence of feasible points of problem~\eqref{eq:SIP}.

For $k \in \mathbb{N}_{+}$, we consider two finite sequences $x^{1}, \dots, x^{k-1} \in \mathcal{X}$, and $v_{1}, \dots, v_{k-1} \in \mathbb{R}$ s.t.\ $v_\ell$ is an upper bound on $\mathsf{val}(\mathsf{P}_{x^{\ell}})$, given by a $\delta$-oracle. Since, for all $\ell = 1, \dots, k-1$, the inequality $-\frac{1}{2}y^\top Q(x^{\ell})y +q(x^{\ell})^\top y + b(x^\ell) \leq v_\ell,$ (i.e., $\langle \mathbf{\mathcal{Q}}(x^\ell), Y \rangle  \leq  v_\ell$ for $Y:=yy^\top$) holds for any $y \in \mathcal{Y}$, the following SDP problem is still a relaxation of \eqref{eq:QCQP}, for every $x \in \mathcal{X}$:

\begin{equation}\label{eq:SDPk_relax}
\left\{  \begin{array}{cll}
     \max \limits_{Y \in \mathbb{S}_{n+1}} & \langle  \mathbf{\mathcal{Q}}(x), Y \rangle  & \\
     \text{s.t.} & \langle \mathbf{\mathcal{Q}}^j, Y \rangle  \leq  0 &  \forall j \in \{ 1, \dots, r \} \\
     & \langle \mathbf{\mathcal{Q}}(x^\ell), Y \rangle  \leq  v_\ell &  \forall \ell \in \{ 1, \dots, k-1 \} \\
     & \langle I_{n+1}, Y \rangle \leq  1+\rho^2 & \\
     & \langle E, Y \rangle = 1 & \\
     & Y \succeq  0. &
\end{array} \tag{\mbox{$\mathsf{SDP}^k_x$}} \right.
\end{equation}
We recall that the value of \eqref{eq:QCQP} is $\phi(x) = \max_{y \in \mathcal{Y}} G(x,y)$. For ease of reading, we also denote by $\phi_\mathsf{SDP}(x)$ the value of \eqref{eq:SDP_relax}, and by $\phi_\mathsf{SDP}^k(x)$ the value of \eqref{eq:SDPk_relax}. We underline that function $\phi_\mathsf{SDP}^k(x)$ implicitly depends on the sequences $(x^\ell)_{1\leq \ell \leq k}$ and $(v_\ell)_{1\leq \ell \leq k}$. Given the Lagrangian multiplier $\zeta_\ell$ associated to the constraint $\langle \mathbf{\mathcal{Q}}(x^\ell), Y \rangle \leq v_\ell$, the strong SDP dual of problem~\eqref{eq:SDPk_relax} is
\begin{equation}
\left\{ \begin{array}{cl}
     \min \limits_{\lambda, \alpha, \beta, \zeta} &  \alpha (1+\rho^2) + \beta + \sum_{\ell = 1}^{k-1} \zeta_\ell v_\ell\\
     \text{s.t.} &  \sum\limits_{j=1}^r  \lambda_j \mathbf{\mathcal{Q}}^j +  \sum_{\ell = 1}^{k-1} \zeta_\ell \mathbf{\mathcal{Q}}(x^\ell) + \alpha I_{n+1} + \beta E \succeq \mathbf{\mathcal{Q}}(x) \\
     & \lambda \in \mathbb{R}^r_+, \: \alpha \in \mathbb{R}_+, \beta \in \mathbb{R}, \zeta \in \mathbb{R}^{k-1}_+
\end{array}
\tag{\mbox{$\mathsf{DSDP}^k_x$}} \right.
\label{eq:DSDPk}
\end{equation}
Applying the analog of Lemma~\ref{lem:strong_duality} to the primal-dual pair of \eqref{eq:SDPk_relax}--\eqref{eq:DSDPk}, for every $\hat{x} \in \mathcal{X}$, $\phi_\mathsf{SDP}^k(\hat{x}) \leq 0$ holds if and only if $\hat{x} \in \mathcal{R}^k$, where $\mathcal{R}^k$ is defined as
\begin{align*}
  \mathcal{R}^k = \left \lbrace \hat{x} \in \mathbb{R}^m \colon    \exists  (\lambda,\alpha,\zeta) \in \mathbb{R}^{r+k}_+,  \exists \beta \in \mathbb{R},  \: \left( \alpha (1+\rho^2) + \beta + \sum_{\ell = 1}^{k-1} \zeta_\ell v_\ell \leq 0 \right) \right. \\  \hspace{2cm}\left. \wedge \left( \sum\limits_{j=1}^r  \lambda_j \mathbf{\mathcal{Q}}^j +  \sum_{\ell = 1}^{k-1} \zeta_\ell \mathbf{\mathcal{Q}}(x^\ell) + \alpha I_{n+1} + \beta E \succeq \mathbf{\mathcal{Q}}(\hat{x}) \right) \right \rbrace .
\end{align*}

\begin{proposition} \label{prop:phik}
Under Assumption~\ref{as:subpbqcqp}, for every finite sequences $x^{1}, \dots, x^{k-1} \in \mathcal{X}$, and $v_{1}, \dots, v_{k-1} \in \mathbb{R}$ s.t.\ $v_\ell \geq \phi(x^{\ell})$, the resulting set $\mathcal{X} \cap \mathcal{R}^k$ is included in the feasible set of \eqref{eq:SIP}.
\end{proposition}
\begin{proof}
    We apply Theorem~\ref{th:finiteformulation} to the inner problem modified with the additional valid cuts $\langle \mathbf{\mathcal{Q}}(x^\ell), Y \rangle  \leq  v_\ell \; \forall \ell \in \{1,\dots,k-1\}$, which do not change the value of the non-convex inner problem $\phi(x)$, and, therefore, do not change the feasible set of problem~\eqref{eq:SIP}. We deduce that $\mathcal{X} \cap \mathcal{R}^k$ is a subset of the feasible set of \eqref{eq:SIP}. 
\end{proof}

\begin{algorithm}[h!]
{\small
\caption{IOA algorithm with inexact oracle}
\label{alg:IO}
\begin{algorithmic}[1]
    \State{\textbf{Input:} Oracle with parameter $\delta \in [0,1)$, tolerances $(\epsilon_1, \epsilon_2) \in \mathbb{R}_+^2$, bounds $\underline{\mu}, \overline{\mu} \in \mathbb{R}_{++}$.}
    \State{Solve the restriction~\eqref{eq:SIP:convexreform}, to obtain $\hat{x}^0$.}
    \If{$Q(\hat{x}^0)\succ 0$ and $Q^1, \dots, Q^r \succeq 0$} \label{step:suffcond}
    \State Return $\hat{x}^0$.
    \EndIf
    \State{$k\gets0$, $\mathcal{Y}^0 \gets \emptyset$, $(\nu^0_1, \nu^0_2) \gets (\infty,\infty)$.}
    \While{$\nu^k_1  > \epsilon_1$ or $\nu^k_2  > \epsilon_2$}
        \State Choose $\mu_k \in [\underline{\mu}, \overline{\mu}]$, and compute an optimal solution $(x^k,\hat{x}^k)$ of 
        \begin{align}
            \left\{ \begin{array}{rl}\label{eq:master_problem_IO}
          \min\limits_{x, \hat{x} \in \mathcal{X}} & F(x) + F(\hat{x}) + \frac{\mu_k}{2} \lVert x - \hat{x}\rVert^2 \\
                \text{s.t.} &   G(x,y) \leq 0 \quad \forall y \in \mathcal{Y}^k \\
                & \hat{x} \in \mathcal{R}^k.
        \end{array}\right.
        \end{align}
	\State Call the $\delta$-oracle to compute an approximate solution $y^{k} \in \mathcal{Y} = \hat{y}(x^k)$, and an upper bound $v_k = \hat{v}(x^k)$ of $\max_{y \in \mathcal{Y}} G(x^k,y)$.
        \State Based on $x^k$ and $v_k$, update the set $\mathcal{R}^k$ in $\mathcal{R}^{k+1}$. 
        \State {$\mathcal{Y}^{k+1} \gets \mathcal{Y}^{k} 
        \cup \{ y^k \} $, $\; (\nu^{k+1}_1,\nu^{k+1}_2) \gets (G(x^k,y^k), \lVert x^k -\hat{x}^k \rVert)$}
	\State {$k \gets k + 1$}
    \EndWhile
    \State  Return $(x^k,\hat{x}^k)$.    \label{steptermIOA}
\end{algorithmic}}
\end{algorithm}

Alg.~\ref{alg:IO} is the pseudocode of the IOA algorithm with inexact oracle. It starts by solving the restriction \eqref{eq:SIP:convexreform} and checks whether the condition presented in Theorem~\ref{th:restriction_optimality} is satisfied or not. If yes, the algorithm stops returning the solution which is optimal for both \eqref{eq:SIP:convexreform} and \eqref{eq:SIP}. Otherwise, it performs a sequence of iterations, until the stopping criteria are satisfied, i.e., $G(x^k,y^k) \leq \epsilon_1$, and $\lVert x^k - \hat{x}^k \rVert \leq \epsilon_2$. At each iteration, the convex optimization problem \eqref{eq:master_problem_IO} is solved. This problem is a \textit{coupling} between the minimization of $F$ on a relaxed set, and the minimization of $F$ on a restricted set. Indeed, $x$ belongs to an outer-approximation (relaxation), whereas $\hat{x}$ belongs to an inner-approximation (restriction) of \eqref{eq:SIP} feasible set. The minimization of $F$ over these two sets is coupled by a proximal term that penalizes the distance between $x$ and $\hat{x}$. After solving the master problem~\eqref{eq:master_problem_IO}, the lower-level problem~\eqref{eq:QCQP} is solved for $x = x^k$. The {solution} of this problem is used to restrict the outer-approximation, and to enlarge the inner-approximation. 

Before proving the termination and the convergence of Alg.~\ref{alg:IO}, under Assumptions~\ref{as:convex} and \ref{as:subpbqcqp}, we introduce two technical lemmas.
\begin{lemma}
    Under Assumptions~\ref{as:convex} and \ref{as:subpbqcqp}, denoting by $x^*$ an optimal solution of \eqref{eq:SIP}, if Alg.~\ref{alg:IO} runs iteration $k$, $F(x^k) \leq F(x^*) + \mu_k(x^k-\hat{x}^k)^\top (x^* - x^k)$. \label{lem:DirOpt}
\end{lemma}
\begin{proof}
Proof in Appendix~\ref{app:DirOpt}.
\end{proof}
\begin{lemma} \label{lem:discretizationconv}
Consider a parameter $\delta \in [0,1)$, the infinite sequences $(x^k)_{k \in \mathbb{N}} \subset \mathcal{X}$ and $(y^k)_{k \in \mathbb{N}}\subset \mathcal{Y}$, where $y^k = \hat{y}(x^k)$ is the output of the $\delta$-oracle evaluated at point $x^k$. If these sequences are such that, for every $k \in \mathbb{N}$,\vspace{-0.7em}
    \begin{align*}
        (*) \quad G(x^k,y^\ell) \leq 0, \quad \forall \ell \in \{0, \dots ,k - 1 \}, \vspace*{-0.7em}
    \end{align*}
then, the feasibility error $\phi(x^k)^+$ vanishes in the limit $k\to \infty$.
\end{lemma}
\begin{proof}
    Proof in Appendix~\ref{app:discretizationconv}.
\end{proof}
Before analyzing the convergence and the termination of the IOA algorithm (Theorems~\ref{th:asymptoticConvIOA}--\ref{th:termIOA-gap}), we make the additional assumption that $F$ is Lipschitz continuous, and that the Slater condition holds for the restriction~\eqref{eq:SIP:convexreform}. However, the knowledge of this Slater point is not a prerequisite for running the IOA algorithm.
\begin{assumption}\label{as:slater_res}
$F(x)$ is $C_F$-Lipschitz, and there exists $x^S \in \mathcal{X}$ s.t.\ $\phi_{\mathsf{SDP}}(x^S)< 0$. 
\end{assumption}
Theorem~\ref{th:asymptoticConvIOA} states the asymptotic convergence of Alg.~\ref{alg:IO} in the case it does not terminate.
\begin{theorem}\label{th:asymptoticConvIOA}
Under Assumptions~\ref{as:convex}, \ref{as:subpbqcqp}, and \ref{as:slater_res}, 
if Alg.~\ref{alg:IO} does not stop then it generates infinite sequences $(x^k)_{k \in \mathbb{N}_{+}}$ and $(\hat{x}^k)_{k \in \mathbb{N}_{+}}$ such that (i) $\lVert x^k - \hat{x}^k \rVert \rightarrow 0$, and (ii) the sequence $(\hat{x}^k)_{n\in\mathbb{N}}$ of feasible points in \eqref{eq:SIP} is a minimizing sequence, i.e., $F(\hat{x}^k)\rightarrow \mathsf{val}\text{\eqref{eq:SIP}}$.

\end{theorem}
\begin{proof} 
We start by proving that $\phi_{\mathsf{SDP}}^k(x^k)^+ \rightarrow 0$. Since $\phi_{\mathsf{SDP}}^k(x^k)^+$ is bounded, there exists at least one accumulation value for this sequence. 
We show that, necessarily, every accumulation point of $\phi_{\mathsf{SDP}}^k(x^k)^+$ is $0$. We take any strictly increasing application $\psi : \mathbb{N} \to \mathbb{N}$, such that $\phi_{\mathsf{SDP}}^{\psi(k)}(x^{\psi(k)})^+ \rightarrow h$. We show that for every $\psi$, $h=0$. From a corollary of Bolzano--Weierstrass theorem \cite[Ex.~2.5.4]{abbott}, we will then conclude that $\phi_{\mathsf{SDP}}^k(x^k)^+ \rightarrow 0$.
    
    Up to the extraction of a subsequence, we can assume, by the compactness of $\mathcal{X}$, that $x^{\psi(k)} \rightarrow x \in \mathcal{X}$. 
    For $k \in \mathbb{N}$, we define $j = \psi(k)$ and $\ell = \psi(k-1)$, and   \vspace{-0.5em}
    \begin{align}
        \phi_{\mathsf{SDP}}^j(x^j) = &\; \phi_{\mathsf{SDP}}^j(x^{\ell}) + \phi_{\mathsf{SDP}}^j(x^j) - \phi_{\mathsf{SDP}}^j(x^{\ell}) \\
                & \leq v_{\ell} + \phi_{\mathsf{SDP}}^j(x^j) - \phi_{\mathsf{SDP}}^j(x^{\ell}),  \label{eq:boundingvl}
    \end{align}
    as the constraint $\langle \mathcal{Q}(x^{\ell}), Y \rangle \leq v_{\ell}$ is enforced in the problem $(\mathsf{SDP}^{j}_{x^{\ell}})$, since $\ell \leq j-1$. 
    We introduce $\tilde{Y}$ the solution of $\mathsf{SDP}^{j}_x$ at $x = x^{j}$ so that, $\langle \mathcal{Q}(x^{j}), \tilde{Y} \rangle =  \phi_{\mathsf{SDP}}^j(x^{j}) $. Since $\tilde{Y}$ is feasible in $\mathsf{SDP}^{j}_x$ at $x = x^{\ell}$, $\langle \mathcal{Q}(x^{\ell}), \tilde{Y} \rangle \leq \phi_{\mathsf{SDP}}^j(x^{\ell})$. Therefore, by linearity of $\mathcal{Q}$, and due to the Cauchy-Schwartz inequality,  $\phi_{\mathsf{SDP}}^j(x^j) - \phi_{\mathsf{SDP}}^j(x^{\ell}) \leq \langle \mathcal{Q}(x^j - x^{\ell}), \tilde{Y} \rangle \leq \lVert  \mathcal{Q}(x^j - x^{\ell}) \rVert_F \: B $ where $B := \max_{Y \in \mathsf{Feas}\text{\eqref{eq:SDP_relax}}} \lVert Y \rVert_F$, which is independent from $j,$ and $\ell$. Indeed, $\tilde{Y}$ is feasible in \eqref{eq:SDP_relax}. By defining the operator norm of $x \mapsto \mathcal{Q}(x)$ as $\lVert \mathcal{Q} \rVert_{op}$, we obtain that $\phi_{\mathsf{SDP}}^j(x^j) - \phi_{\mathsf{SDP}}^j(x^{\ell}) \leq \lVert x^j - x^\ell \rVert \: \lVert \mathcal{Q} \rVert_{op} \: B$. We combine this with Eq.~\eqref{eq:boundingvl}, using the fact that the positive part is non-decreasing, to obtain\vspace{-0.7em}
    \begin{align}
    \phi_{\mathsf{SDP}}^j(x^j)^+  & \leq v_{\ell}^+ + \lVert x^j - x^\ell \rVert \: \lVert \mathcal{Q} \rVert_{op} \: B \\
    & \leq \phi(x^\ell)^+ (1+ \delta) + \lVert x^j - x^\ell \rVert \: \lVert \mathcal{Q} \rVert_{op} \: B,
\end{align}
the second inequality coming from the property of the $\delta$-oracle (see Eq.~\eqref{eq:oraclebound}). Using the definition of $\ell$ and $j$, we obtain $ 0 \leq \phi_{\mathsf{SDP}}^{\psi(k)}(x^{\psi(k)})^+ \leq \phi(x^{\psi(k-1)})^+ (1+ \delta) + \lVert x^{\psi(k)} - x^{\psi(k-1)} \rVert \: \lVert \mathcal{Q} \rVert_{op} \: B.$
Since $\phi(x^{k})^+ \rightarrow 0$ (due to Lemma~\ref{lem:discretizationconv}), and $x^{\psi(k)}$ is converging, we deduce, by taking the limit, that $h = 0$. This holds for any $\psi$, i.e., for any converging subsequence $\phi_{\mathsf{SDP}}^{\psi(k)}(x^{\psi(k)})^+$. We can thus conclude that $\phi_{\mathsf{SDP}}^k(x^k)^+ \rightarrow 0$ \cite[Ex.~2.5.4]{abbott}.

Using Assumption~\ref{as:slater_res}, we introduce a Slater point $x^S \in \mathcal{X}$ such that $\phi_{\mathsf{SDP}}(x^S) = -c$, for $c > 0$. We also introduce $\omega_k := \phi_{\mathsf{SDP}}^k(x^k)^+/(c+\phi_{\mathsf{SDP}}^k(x^k)^+)$. We notice that $\omega_k \rightarrow 0$, since $\phi_{\mathsf{SDP}}^k(x^k)^+ \rightarrow 0$. We define the convex combination $\bar{x}_k = (1-\omega_k) x^k + \omega_k x^S$. We emphasize that $(\bar{x}_k,\bar{x}_k)$ is feasible in problem~\eqref{eq:master_problem_IO} at iteration $k$ since $\mathcal{X}$ is convex, and 
    \begin{itemize}
\item $\bar{x}_k$ satisfies the constraints on $x$, because both $x^k$ and $x^S$ satisfy the convex constraints $G(x,y)$ for $y \in \mathcal{Y}^k$, and, by convex combination, so does $\bar{x}_k$;
    \item $\bar{x}_k$ satisfies the constraints on $\hat{x}$;  indeed, by convexity of $\phi_{\mathsf{SDP}}^k(x)$ (as a maximum of linear functions), the following holds:
     {\small\begin{align}
        \phi_{\mathsf{SDP}}^k(\bar{x}_k) &\leq (1- \omega_k) \phi_{\mathsf{SDP}}^k(x^k) + \omega_k \phi_{\mathsf{SDP}}^k(x^S) \\
        & \leq \frac{c}{c+\phi_{\mathsf{SDP}}^k(x^k)^+} \phi_{\mathsf{SDP}}^k(x^k) + \frac{\phi_{\mathsf{SDP}}^k(x^k)^+}{c+\phi_{\mathsf{SDP}}^k(x^k)^+} \phi_{\mathsf{SDP}}^k(x^S).
    \end{align}}
    Since $\phi_{\mathsf{SDP}}^k(x^S) \leq \phi_{\mathsf{SDP}}(x^S) = - c$ and $\frac{\phi_{\mathsf{SDP}}^k(x^k)^+}{c+\phi_{\mathsf{SDP}}^k(x^k)^+} \geq 0$, we obtain
    {\small\begin{align} 
        \phi_{\mathsf{SDP}}^k(\bar{x}_k)\leq \frac{c}{c+\phi_{\mathsf{SDP}}^k(x^k)^+}\phi_{\mathsf{SDP}}^k(x^k) - \frac{c}{c+\phi_{\mathsf{SDP}}^k(x^k)^+}\phi_{\mathsf{SDP}}^k(x^k)^+,
    \end{align}}
    and thus
    {\small\begin{align} 
        \phi_{\mathsf{SDP}}^k(\bar{x}_k)\leq \frac{c}{c+\phi_{\mathsf{SDP}}^k(x^k)^+} \left(\phi_{\mathsf{SDP}}^k(x^{k}) - \phi_{\mathsf{SDP}}^k(x^{k})^+\right).
    \end{align}}
    
    As $\frac{c}{c+\phi_{\mathsf{SDP}}^k(x^k)^+} \geq 0$ and $\phi_{\mathsf{SDP}}^k(x^{k}) \leq \phi_{\mathsf{SDP}}^k(x^{k})^+$, we deduce that $\phi_{\mathsf{SDP}}^k(\bar{x}_k) \leq 0$, i.e., $\bar{x}_k \in \mathcal{R}^k$.
\end{itemize}
As the objective value of $(\bar{x}_k,\bar{x}_k)$ in the problem~\eqref{eq:master_problem_IO} is $2F(\bar{x}_k)$, by optimality of $(x^k, \hat{x}^k)$: $ F(x^k) +  F(\hat{x}^k) + \frac{\mu_k}{2} \lVert x^k - \hat{x}^k \rVert^2  \leq 2 F( (1-\omega_k) x^k + \omega_k x^S)$, which means, by convexity of $F$, that \vspace{-1em}
\begin{align}
   F(x^k) +  F(\hat{x}^k) + \frac{\mu_k}{2} \lVert x^k - \hat{x}^k \rVert^2  \leq 2(1-\omega_k) F(x^k) + 2\omega_k F(x^S).
\label{eq:start_point_slater}
\end{align}
We also notice that $(\hat{x}^{k},\hat{x}^{k})$ is feasible in the problem~\eqref{eq:master_problem_IO} at iteration $k$, thus $F(x^k) +  F(\hat{x}^k) + \frac{\mu_k}{2} \lVert x^k - \hat{x}^k \rVert^2 \leq  2 F( \hat{x}^{k})$, which means\vspace{-0.7em}
\begin{equation}
     F(x^k)  + \frac{\mu_k}{2} \lVert x^k - \hat{x}^k \rVert^2 \leq  F( \hat{x}^{k}).      \label{eq:observation1}
\end{equation}
Summing Eq.~\eqref{eq:start_point_slater} with Eq.~\eqref{eq:observation1}, we have: $ 2 F(x^k) + \mu_k \lVert x^k - \hat{x}^k \rVert^2 \leq 2(1-\omega_k) F(x^k) + 2\omega_k F(x^S)$, and thus $\mu_k \lVert x^k - \hat{x}^k \rVert^2 \leq 2 \omega_k \left( F(x^S) -  F(x^k) \right)$. Since $0<\underline{\mu} \leq \mu_k$,\vspace{-0.3em}
\begin{equation}
     \lVert x^k - \hat{x}^k \rVert \leq \sqrt{\underline{\mu}^{-1}(2 \omega_k \left( F(x^S) -  F(x^k) \right))}
     \label{eq:observation1bis}
\end{equation}
holds. Since $\omega_k \rightarrow 0$, and $F(x^S) -  F(x^k)$ is bounded, we deduce from Eq.~\eqref{eq:observation1bis} that \vspace{-0.7em}
\begin{equation}
     \lVert x^k - \hat{x}^k \rVert \rightarrow 0.
     \label{eq:almost_conclusion}
\end{equation}

Every point $\hat{x}^k \in \mathcal{R}^{k}$ of the generated sequence is feasible in \eqref{eq:SIP}, as stated in Proposition~\ref{prop:phik}, and $F$ is $C_F$-Lipschitz. Therefore, we have $F(x^*) \leq F(\hat{x}^k) \leq F(x^k) + C_F \lVert x^k - \hat{x}^k\rVert$. According to Lemma~\ref{lem:DirOpt}, we know that $F(x^k) \leq F( x^{*}) + \mu_k (x^k - \hat{x}^k)^\top(x^* - x^k)$, which implies, according to the Cauchy-Schwartz inequality, that 
\begin{equation}
    F(x^*) \leq F(\hat{x}^k) \leq F(x^*) + \mu_k \lVert x^k - \hat{x}^k\rVert \lVert x^* - x^k \rVert + C_F \lVert x^k - \hat{x}^k\rVert.
    \label{eq:encadrement}
\end{equation}
Since $\lVert x^* - x^k \rVert$ is bounded, we deduce from Eq.~\eqref{eq:almost_conclusion}  that $F(x^*) + \mu_k \lVert x^k - \hat{x}^k\rVert \lVert x^* - x^k \rVert + C_F \lVert x^k - \hat{x}^k\rVert \rightarrow F(x^*)$, and thus, $F(\hat{x}^k) \rightarrow F(x^*) =  \mathsf{val}\text{\eqref{eq:SIP}}$.
\end{proof}
We move now to the case where the IOA Algorithm terminates in a finite number of iterations: Theorem~\ref{th:finiteTermIO} studies a sufficient condition for such a finite termination, and Theorem~\ref{th:termIOA-gap} studies the suboptimality of the returned feasible point.
\begin{theorem}\label{th:finiteTermIO}
Under Assumptions~\ref{as:convex}, \ref{as:subpbqcqp}, and \ref{as:slater_res}, if $\epsilon_1>0$ and $ \epsilon_2 > 0$ then Alg.~\ref{alg:IO} terminates after a finite number of iterations.
 \end{theorem}
\begin{proof}
We reason by contrapositive: we suppose that Alg.~\ref{alg:IO} does not stop, i.e., generates two infinite sequences $(x^k)_{k \in \mathbb{N}_{+}}$ and $(\hat{x}^k)_{k \in \mathbb{N}_{+}}$, and we show that $\epsilon_1 =0$ or $\epsilon_2 = 0$. In the one hand, as the stopping criterion is not met, for all $k \in\mathbb{N}$, $G(x^k,y^k) > \epsilon_1$ (and therefore $G(x^k,y^k)^+ > \epsilon_1$), or $\lVert x^k - \hat{x}^k \rVert > \epsilon_2$. In other words, $\max \{ G(x^k,y^k)^+ - \epsilon_1, \lVert x^k - \hat{x}^k \rVert - \epsilon_2 \} >0$. In the other hand, we deduce from Lemma~\ref{lem:discretizationconv} that $G(x^k,y^k)^+ \leq \phi(x^k)^+ \rightarrow 0$, and from Theorem~\ref{th:asymptoticConvIOA}  that $\lVert x^k - \hat{x}^k \rVert \rightarrow 0$. By continutuity of the max operator, we deduce that $\max \{ - \epsilon_1, - \epsilon_2 \} \geq 0$, i.e. $\min \{ \epsilon_1, \epsilon_2 \} \leq 0$. The scalars $\epsilon_1$ and $\epsilon_2$ being nonnegative by definition, we deduce that $\min \{ \epsilon_1, \epsilon_2 \} = 0$.
\end{proof}
 
Theorem~\ref{th:termIOA-gap} provides an upper bound on the optimality gap of the returned feasible point in case of finite termination. We use the notation $\mathsf{diam}(\mathcal{X}) := \underset{x_1, x_2 \in \mathcal{X}}{\max} \lVert x_1 - x_2\rVert$.
\begin{theorem} \label{th:termIOA-gap}
    Under Assumptions~\ref{as:convex}, \ref{as:subpbqcqp}, and \ref{as:slater_res}, if Alg.~\ref{alg:IO} terminates after iteration $K$, then it returns a feasible iterate $\hat{x}^K$ s.t.\ $F(\hat{x}^K) \leq \mathsf{val}\text{\eqref{eq:SIP}} + \epsilon_2 (\mu_K \mathsf{diam}(\mathcal{X}) + C_F).$
\end{theorem}
\begin{proof}
From Lemma~\ref{lem:DirOpt}, and from Cauchy-Schwartz inequality, we know that $F(x^K) \leq F(x^*) + \mu_k(x^K-\hat{x}^K)^\top (x^* - x^K) \leq F(x^*) + \mu_k \lVert x^K-\hat{x}^K \rVert \: \lVert x^* - x^K \rVert$. Using that $F$ is $C_F$-Lipschitz, we obtain  $F(\hat{x}^K) \leq F(x^*) + \mu_K \lVert x^K-\hat{x}^K \rVert \: \lVert x^* - x^K \rVert + C_F \lVert x^K-\hat{x}^K \rVert$. We note that $\lVert x^K - \hat{x}^K\rVert \leq \epsilon_2$ and $\lVert x^* - x^K \rVert \leq  \mathsf{diam}(\mathcal{X})$ to conclude.
\end{proof}

\section{Conclusion}
In this paper, we address the issue of solving a convex Semi-Infinite Programming problem despite the difficulty of the separation problem. We proceed by allowing for the approximate solution of the separation problem, up to a given relative optimality gap. We see that, in the case where the objective function is strongly convex, the Cutting-Planes algorithm has guaranteed theoretical performance despite the inexactness of the oracle: it converges in $O(1/k)$. Contrary to the Cutting-Planes algorithm, the Inner-Outer Approximation algorithm generates a sequence of feasible points converging towards the optimum of the Semi-Infinite Programming problem. This paper shows that this is also the case when the separation problem is a Quadratically Constrained Quadratic Programming problem despite the inexactness of the oracle. An avenue of research is to extend these results to the setting of Mixed-Integer Convex Semi-Infinite Programming.
\vspace{1em}

\noindent
{\small\textbf{Data availability statement:} 
Data sharing not applicable to this article as no datasets were generated or analyzed during the current study.}

\noindent
{\small\textbf{Acknowledgment:} 
The research of M.\ Cerulli was partially supported by project ``SEcurity and RIghts in the CyberSpace'' SERICS (PE00000014) under the MUR National Recovery and Resilience Plan funded by the European Union - NextGenerationEU.}

\bibliography{convex-sip-inexact}

\begin{appendices}
\section{}
\subsection{Notations}\label{app:notation}
We summarize the notations used throughout the paper in Table~\ref{tab:notation}.
\begin{table}[!ht]
    \centering
    \scalebox{0.9}{\begin{tabular}{|l|l|}
    \hline
        \rowcolor[HTML]{C0C0C0}\textbf{Symbol} & \textbf{Description} \\ \hline
        \rowcolor{black!10} $\mathbb{R} = (-\infty,+\infty)$ & the set of real numbers \\ \hline
        $\mathbb{R}_+ = [0,+\infty)$ & the set of non-negative real numbers \\ \hline
        \rowcolor{black!10} $\mathbb{R}_{++} = (0,+\infty)$ & the set of positive real numbers \\ \hline
        $\mathbb{R}^n$ & set of real $n$-dimensional vectors \\ \hline
        \rowcolor{black!10} $\mathbb{Z} = \{\dots,-2,-1,0,1,2,\dots\}$ & the set of integer numbers \\ \hline
        $\mathbb{N} = \{0,1,2,3,\dots\}$ & the set of natural numbers including $0$ \\ \hline
        \rowcolor{black!10} $\mathbb{N}^+ = \{1,2,3,\dots\}$ & the set of natural numbers without $0$ \\ \hline
        $\mathbb{S}_n$ & the set of symmetric $n \times n$ matrices \\ \hline
        \rowcolor{black!10} $\mathbb{S}_n^+$ & the set of positive semidefinite $n \times n$ matrices \\ \hline
        $\langle A,B \rangle = \mathsf{Tr}(A^\top B)$ & the Frobenius inner product of two square matrices $A$ and $B$ with the same size \\ \hline
        \rowcolor{black!10} $\mathsf{conv}(S)$ & the convex-hull of the set $S,$ i.e., the smallest convex set that contains $S$ \\ \hline
        \rowcolor{black!10} $\mathsf{cone}(S)$ & the conic-hull of the set $S,$ i.e., the smallest convex cone that contains $S$ \\ \hline
        $\mathsf{diag}(a)$ & the diagonal $n\times n$ matrix having vector $a = (a_1,\dots,a_n)$ as diagonal \\ \hline
        \rowcolor{black!10}  $\lVert a \rVert$ & the Euclidean norm of vector $a = (a_1,\dots,a_n)$, i.e. $\sqrt{a_1^2 + \dots + a_n^2}$ \\ \hline
    \end{tabular}}\label{tab:notation}
    \caption{List of symbols.}
\end{table}

\subsection{Proof of Lemma~\ref{lem:smoothness}}\label{app:smoothness}
We have that: \begin{inparaenum}[(i)] \item the set $\mathcal{X}$ is compact, \item the function $-\mathcal{L}(\cdot,z)$ is continuous for all $z \in \mathbb{R}^m$, \item the function $-\mathcal{L}(x,\cdot)$ is convex and differentiable for all $x \in \mathcal{X}$, \item the function  $\sup_{x\in \mathcal{X}} -\mathcal{L}(x,\cdot)$ (with $x$ function of $z$) is finite-valued over $\mathbb{R}^m$, and \item due to Assumptions~\ref{as:convex} and \ref{as:strongconvex}, the function $-\mathcal{L}(\cdot,z)$ is strongly concave and the set $\mathcal{X}$ is compact and convex, therefore the supremum $\sup_{x\in \mathcal{X}} -\mathcal{L}(x,z)$ is attained for a unique $x(z)$. \end{inparaenum} 
In this setting, we deduce from \cite[Cor.~VI.4.4.5]{hiriart2013convex} that $\theta(z)$ is differentiable over $\mathbb{R}^m$, with gradient $\nabla_z \mathcal{L}(x(z),z) = x(z)$. 

We now take $z,z' \in \mathbb{R}^m$, and prove that $\lVert \nabla \theta(z) - \nabla \theta(z') \rVert \leq \frac{1}{\mu} \lVert z - z'\rVert$. We define the functions $w(u) = \mathcal{L}(u,z) + {i}_\mathcal{X}(u)$ and $w'(u) = \mathcal{L}(u,z') + {i}_\mathcal{X}(u)$, where $i_\mathcal{X}(\cdot)$ is the characteristic function of $\mathcal{X}$.  We introduce $x$ (resp. $x'$) the unique minimum of $w$ (resp. $w'$). The first-order optimality condition for these convex functions reads \vspace{-0.8em}
\begin{equation}
 0 \in \partial w (x),\quad 0 \in \partial w' (x').  \label{eq:foc} \vspace{-0.5em}
\end{equation}
We notice that the function $(F + i_\mathcal{X})(u)$ is convex due to Assumption~\ref{as:convex}, and the function $\ell(u) = z^\top u$ is linear and thus convex. The intersection of the relative interiors of the domains of
these convex functions is $\mathsf{ri}(\mathcal{X})$. With $\mathcal{X}$ a finite-dimensional convex set, $\mathsf{ri}(\mathcal{X}) \neq \emptyset$, according to \cite[Prop.~1.9]{tuy}. Hence, the subdifferential of the sum is the sum of the subdifferentials \cite[Th.~23.8]{rockafellar1970}, i.e., $\partial w(x) = \partial (F + i_\mathcal{X})(x) + \partial \ell (x) =\partial (F + i_\mathcal{X})(x) + z$. Similarly,  $\partial w'(x') = \partial (F + i_\mathcal{X})(x') +  z'$. Therefore, Eqs.~\eqref{eq:foc} may be rephrased as the existence of $s \in \partial (F + i_\mathcal{X})(x)$ and $s' \in \partial (F + i_\mathcal{X})(x')$ such that \vspace{-0.8em}
\begin{equation}
 0 = s + z, \quad 0 = s'+z'. \label{eq:subdiff} \vspace{-0.5em}
\end{equation}
Due to Assumptions~\ref{as:convex} and \ref{as:strongconvex}, the function $F + i_\mathcal{X}$ is $\mu$-strongly convex. Applying  \cite[Th.~VI.6.1.2]{hiriart2013convex}, the $\mu$-strong convexity of $F + i_\mathcal{X}$ gives that $     (s-s')^\top(x -x' ) \geq \mu \lVert x - x' \rVert^2$, since $s \in \partial (F + i_\mathcal{X})(x)$ and $s' \in \partial (F + i_\mathcal{X})(x')$. Using the Cauchy-Schwartz inequality and Eqs.~\eqref{eq:subdiff}, we deduce that $\lVert z - z' \rVert \: \lVert x - x' \rVert \geq \mu \lVert x - x' \rVert^2$. 
Since $\nabla \theta (z) = x$ and  $\nabla \theta (z') = x'$ (following from the first part of this proof):
\begin{align}
    \lVert z - z' \rVert \: \lVert \nabla\theta(z) - \nabla\theta(z') \rVert  \geq \mu \lVert \nabla\theta(z) - \nabla\theta(z') \rVert^2. \label{eq:almostconclude}
\end{align}
From Eq.~\eqref{eq:almostconclude}, $ \lVert \nabla\theta(z) - \nabla\theta(z') \rVert \leq \frac{1}{\mu} \lVert z - z' \rVert$, if $\lVert \nabla\theta(z) - \nabla\theta(z') \rVert > 0$. If $\lVert \nabla\theta(z) - \nabla\theta(z') \rVert =0$, this inequality is also trivially true.

\subsection{Proof of Lemma~\ref{lem:dualopti}}\label{app:dualopti}
According to Assumption~\ref{as:slater}, we introduce $\hat{x} \in \mathcal{X}$ such that $\hat{x}^\top a(y) < 0$ for all $y \in \mathcal{Y}$. By continuity of $G$ and compactness of $\mathcal{Y}$, we know that there exists a constant $c > 0$ such that $\hat{x}^\top a(y) \leq -c$ for all $y \in \mathcal{Y}$. For every $z \in \mathcal{K}$, there exist $r \in \mathbb{N}$, $y_1, \dots, y_r \in \mathcal{Y}$, and $\lambda_1, \dots, \lambda_r \in \mathbb{R}_{++}$ such that $z = \sum_{i=1}^p \lambda_i a(y_i)$. By definition of $\theta(z)$, $\theta(z) \leq \mathcal{L}(\hat{x},z)$. We deduce that $\theta(z) \leq F(\hat{x}) - c \sum_{i=1}^p \lambda_i$, which also reads $\sum_{i=1}^p \lambda_i  \leq c^{-1}(F(\hat{x}) - \theta(z)).$
For this equation, we deduce that for every feasible point of \eqref{eq:dualsip} such that $\theta(z) \geq V -1$ where $V = \mathsf{val} \text{\eqref{eq:SIP}} =  \mathsf{val} \text{\eqref{eq:dualsip}}$, we have $\sum_{i=1}^p \lambda_i  \leq c^{-1}(F(\hat{x}) - V + 1),$ 
and this holds for every decomposition $z =\sum_{i=1}^p \lambda_i a(y_i)$. In particular, the set of $z \in \mathcal{K}$, such that $\theta(z) \geq V -1$ is included in the compact set $q \mathsf{conv}(\mathcal{M})$, where $q = c^{-1}(F(\hat{x}) - V + 1)$. Over this compact set, the continuous function $\theta(z)$ reaches a maximum.

\subsection{Proof of Lemma~\ref{lem:strongdualityk}}\label{app:strongdualityk}
We notice that the problems pair \eqref{eq:Rk}-\eqref{eq:Dk} is a particular case of the generic problems pair \eqref{eq:lsip}-\eqref{eq:dualsip} when instantiating  ${\mathcal{M}} \gets \mathcal{M}^k$. The equality $\mathsf{val}\eqref{eq:Rk} = \mathsf{val}\eqref{eq:Dk}$ is thus the application of Eq.~\eqref{eq:duality} in this case. The existence of an optimal solution $x^k$ in \eqref{eq:Rk} follows from the compactness of $\mathcal{X}$, the continuity of $F(x)$, and the existence of a feasible point due to Assumption~\ref{as:slater}. The existence of an optimal solution $z^k$ to \eqref{eq:Dk} is a direct application of Lemma~\ref{lem:dualopti}, which is applicable since \eqref{eq:Rk}-\eqref{eq:Dk} also satisfies Assumption~\ref{as:slater}. Indeed, $\hat{x}^\top z < 0$ for all $z \in \mathcal{M}$, so $\hat{x}^\top z < 0$ for all $z \in \mathsf{conv}(\mathcal{M})$; as $\mathcal{M}^k \subset \mathsf{conv}(\mathcal{M})$ by construction, we deduce that $\hat{x}^\top z < 0$ for all $z \in \mathcal{M}^k$.
By definition of $\theta$, $\mathcal{L}(x,z^k) \geq \theta(z^k)$ for all $x \in \mathcal{X}$. We also know that $\mathcal{L}(x^k,z^k) = \theta(z^k)$ since $\mathsf{val}\eqref{eq:Rk} = F(x^k) \geq \mathcal{L}(x^k,z^k) \geq \theta(z^k) = \mathsf{val} \eqref{eq:Dk} = \mathsf{val}\eqref{eq:Rk}$, which means $x^k = \arg\min\limits_{x \in \mathcal{X}} \mathcal{L}(x,z^k)$. Lemma~\ref{lem:smoothness} gives $\nabla \theta (z^{k}) = \arg\min\limits_{x \in \mathcal{X}} \mathcal{L}(x,z^k)$, thus $\nabla \theta (z^{k}) = x^k$.

\subsection{Proof of Lemma~\ref{lem:qcqpsdp}}\label{app:qcqpsdp}
We underline that, due to Assumption~\ref{as:subpbqcqp}, the constraint $\lVert y \rVert^2 + 1 \leq 1 + \rho^2 $ is redundant in the problem \eqref{eq:QCQP}, i.e., adding it does not change the value of this problem. Consequently, we recognize that \eqref{eq:SDP_relax} is the standard Shor SDP relaxation of the problem \eqref{eq:QCQP} augmented with this redundant constraint, this is why $\mathsf{val} \text{\eqref{eq:SDP_relax}} \geq \mathsf{val} \text{\eqref{eq:QCQP}}$. \\
We assume now that $Q(x), Q^1, \dots, Q^r$ are PSD. Given a matrix $Y$ feasible for \eqref{eq:SDP_relax}, we denote by $u_1, \dots, u_{n+1}$ $\in \mathbb{R}^{n+1}$ a basis of eigenvectors of $Y$ (which is PSD) and their respective eigenvalues $v_1, \dots, v_{n+1}$  $\in \mathbb{R}_+$. Let us introduce the two following index sets: $I = \{ i \in \{ 1, \dots, n+1 \} : (u_i)_{n+1} \neq 0 \} \text{ and }  J = \{ i \in\{ 1, \dots, n+1 \} : (u_i)_{n+1} = 0 \}.$ We have then $I \cup J = \{ 1, \dots, n+1 \}$. Moreover, 
\begin{itemize}
\item if $i \in I$ : we define the nonnegative scalar $\mu_i = v_i \: (u_i)_{n+1}^2$ and $y_i \in \mathbb{R}^n$ s.t.\ $u_i = (u_i)_{n+1} \begin{pmatrix} y_i \\ 1
    \end{pmatrix}$ 
    \item if $i \in J$ : we define the nonnegative scalar $\nu_i = v_i$ and $z_i \in \mathbb{R}^n$ s.t.\ $u_i = \begin{pmatrix}
z_i \\
0
\end{pmatrix}$.
\end{itemize}
With this notation, we have that 
{\small$$Y = \sum\limits_{i = 1}^{n+1} v_i u_i u_i^\top  =\sum\limits_{i \in I} \mu_i \begin{pmatrix} y_i y_i^\top & y_i \\ y_i^\top & 1 
    \end{pmatrix} + \sum\limits_{i \in J} \nu_i \begin{pmatrix} z_i z_i^\top & \mathbf{0} \\ \mathbf{0}^\top & 0 
    \end{pmatrix},$$}where $\mathbf{0}$ is the null $n$-dimensional vector (whereas ${0_n}$ is the $n \times n$ null matrix). Let us define the vector $\bar{y} = \sum\limits_{i \in I} \mu_i y_i$. Its objective value in \eqref{eq:QCQP} is larger than the objective value of $Y$ in \eqref{eq:SDP_relax}, since
{\small\begin{align}
     \langle \mathcal{Q}(x) , Y \rangle  &= \sum\limits_{i \in I} \mu_i G(x,y_i) - \frac{1}{2}\sum\limits_{i \in J} \nu_iz_i^\top Q(x) z_i  \\ & \leq \sum\limits_{i \in I} \mu_i G(x,y_i)  \\ &\leq   G(x,\sum\limits_{i \in I} \mu_i y_i) = G(x,\bar{y}).
      \label{eq:step1}
\end{align}}
The first inequality is due to $Q(x) \succeq 0$ and $\nu_i \geq 0$. The second inequality derives from $\sum_{i \in I} \mu_i = Y_{n+1,n+1} = 1$, and from the concavity of the function $G(x,\cdot)$ (Jensen inequality). Similarly, knowing that $Q^j$ is PSD and that $Y$ is feasible in \eqref{eq:SDP_relax}, we can show that $\frac{1}{2} \bar{y}^\top Q^j \bar{y} +  (q^j) ^\top \bar{y} + b_j  \leq \langle \mathbf{\mathcal{Q}}^j, Y \rangle \leq 0$, which means that $\bar{y}$ is feasible in \eqref{eq:QCQP}. This implies that $\langle \mathcal{Q}(x),Y \rangle \leq G(x,\bar{y}) \leq \mathsf{val}\text{\eqref{eq:QCQP}}$. This being true for any matrix $Y$ feasible in \eqref{eq:SDP_relax}, we conclude that $\mathsf{val}\text{\eqref{eq:SDP_relax}}  \geq \mathsf{val} \text{\eqref{eq:QCQP}}$. This proves that $\mathsf{val}\text{\eqref{eq:SDP_relax}}  = \mathsf{val}\text{\eqref{eq:QCQP}}$.

\subsection{Proof of Lemma~\ref{lem:strong_duality}}\label{app:strong_duality}
The Lagrangian of problem \eqref{eq:SDP_relax} is defined over $Y \in  \mathbb{S}_{n+1}^+$, $\lambda \in \mathbb{R}_+^r, \alpha \in \mathbb{R}_+, \beta\in\mathbb{R}$ and reads
{\small$
        L_x(Y, \lambda, \alpha, \beta) = \langle \mathcal{Q}(x) , Y \rangle
         - \sum\limits_{j=1}^r \lambda_j \langle \mathcal{Q}^j, Y \rangle
         + \alpha ( 1 + \rho^2 - \langle I_{n+1}, Y \rangle) 
         + \beta (1 -  \langle E, Y \rangle)  \vspace*{-0.2em} = \alpha (1+\rho^2) + \beta + \langle \mathbf{\mathcal{Q}}(x)  - \sum\limits_{j=1}^r  \lambda_j \mathbf{\mathcal{Q}}^j  - \alpha I_{n+1} - \beta E, Y \rangle.$}
         
The Lagrangian dual problem of \eqref{eq:SDP_relax} is {\small $\min\limits_{\lambda,\alpha,\beta} \, \sup\limits_{Y} \,L_x(Y, \lambda, \alpha, \beta)$}, i.e.,\vspace{-0.8em}
\begin{equation*}
    \underset{\begin{subarray}{c} \lambda \in \mathbb{R}^r_+ \\ \alpha \in \mathbb{R}_+ \\ \beta \in \mathbb{R} \end{subarray}}{\min}  \alpha(1+\rho^2) + \beta + \sup\limits_{Y \in \mathbb{S}_{n+1}^+} \langle \mathbf{\mathcal{Q}}(x)  - \sum\limits_{j=1}^r  \lambda_j \mathbf{\mathcal{Q}}^j  - \alpha I_{n+1} - \beta E, Y \rangle. \vspace{-0.8em}
\end{equation*}

We recognize that the supremum is $+\infty$, unless $\sum\limits_{j=1}^r  \lambda_j \mathbf{\mathcal{Q}}^j  + \alpha I_{n+1} + \beta E \succeq \mathbf{\mathcal{Q}}(x)$. 
This proves that the dual problem of \eqref{eq:SDP_relax} can be formulated as \eqref{eq:DSDP}. We prove now that the Slater condition, which is a sufficient condition for strong duality ($\mathsf{val}\text{\eqref{eq:SDP_relax}} = \mathsf{val}\text{\eqref{eq:DSDP}}$) \cite[p.~265]{boyd2004convex}, holds for the dual problem \eqref{eq:DSDP}. We denote by $m_x$ the maximum eigenvalue of $\mathcal{Q}(x)-\sum\limits_{j=1}^r \mathbf{\mathcal{Q}}^j$, and we notice that $(\lambda,\alpha,\beta) = (1, \dots, 1, \max\{1 + m_x, 1\}, 0)$ is a strictly feasible point of \eqref{eq:DSDP}. Hence, the Slater condition holds.

\subsection{Proof of Lemma~\ref{lem:DirOpt}}\label{app:DirOpt}
We analyze the variation of the objective function w.r.t.\ the variable $x$. Since $x^* \in \mathcal{X}$ is a feasible value for variable $x$, the direction $h = x^* - x^k$ is admissible at $x^k$ in the problem~\eqref{eq:master_problem_IO}. As $F(x)$ is convex over $\mathbb{R}^n$, the directional derivative $F'(x^k,h) = \lim\limits_{t\rightarrow 0^+} \frac{F(x^k + t h) - F(x^k)}{t}$ is well-defined. By optimality of $x^k$, the directional derivative of function $F(x) + \frac{\mu_k}{2} \lVert x - \hat{x}^k \rVert^2$ in the direction $h$ is non-negative, i.e., $F'(x^k,h) + \mu_k (x^k- \hat{x}^k)^\top h \geq 0$. By convexity of $F(x)$, $F(x^*) - F(x^k) \geq F'(x^k,h)$. Combining this with the previous inequality yields $F(x^k) \leq F( x^{*}) + \mu_k (x^k - \hat{x}^k)^\top (x^* - x^k)$.

\subsection{Poof of Lemma~\ref{lem:discretizationconv}}\label{app:discretizationconv}
    Let $t^+ = \max \{ t,0 \}$ be the positive part of function $t$. We notice that the sequence $\phi(x^k)^+$ is bounded, and thus admits at least one accumulation value $\ell$. We are going to prove that $\ell = 0$.
    Let $\psi:\mathbb{N} \to \mathbb{N}$ be any increasing function, such that $\phi(x^{\psi(k)})^+ \rightarrow \ell$. 
    By compactness of $\mathcal{X}$ (resp. $\mathcal{Y}$), we also assume that $x^{\psi(k)} \rightarrow x \in \mathcal{X}$ (resp. $y^{\psi(k)} \rightarrow y \in \mathcal{Y}$). For $(*)$, $G(x^{\psi(k)},y^j) \leq 0$ for all $j \in \{0,\dots, \psi(k) - 1\}$; in particular, $G(x^{\psi(k)},y^{\psi(k-1)}) \leq 0$.
    We deduce that $G(x^{\psi(k)},y^{\psi(k)}) \leq  G(x^{\psi(k)},y^{\psi(k)}) - G(x^{\psi(k)},y^{\psi(k-1)}),$ and therefore, since the positive part of a function is non-decreasing,\vspace{-0.6em}
\begin{align}
     \Bigl( G(x^{\psi(k)},y^{\psi(k)})\Bigr)^+ \leq  \Bigl( G(x^{\psi(k)},y^{\psi(k)}) - G(x^{\psi(k)},y^{\psi(k-1)})\Bigr)^+.
     \label{eq:diffremark}
\end{align}
 According to the definition of the $\delta$-oracle, Eqs.~\eqref{eq:oraclebound} yields $\phi(x^{\psi(k)}) - G(x^{\psi(k)},y^{\psi(k)}) \leq \delta \lvert \phi(x^{\psi(k)}) \rvert$. If $\phi(x^{\psi(k)}) \geq 0$, this means that $ \phi(x^{\psi(k)}) - G(x^{\psi(k)},y^{\psi(k)}) \leq \delta \: \phi(x^{\psi(k)})$, i.e., $\phi(x^{\psi(k)}) \leq \frac{1}{1 - \delta} G(x^{\psi(k)},y^{\psi(k)})$. To also cover the case, $\phi(x^{\psi(k)}) < 0$, we can write $\phi(x^{\psi(k)})^+ \leq \frac{1}{1 - \delta} G(x^{\psi(k)},y^{\psi(k)})^+$.  As $\frac{1}{1-\delta} \geq 0$, we obtain from Eq.~\eqref{eq:diffremark} that $\phi(x^{\psi(k)})^+ \leq \frac{1}{1 - \delta} \Bigl( G(x^{\psi(k)},y^{\psi(k)}) - G(x^{\psi(k)},y^{\psi(k-1)})\Bigr)^+$.   
By continuity of the functions $G$, and $\phi$, and the positive part, we deduce that $\ell = \phi(x)^+ \leq \frac{1}{1 - \delta} \left( G(x,y) - G(x,y)\right)^+ = 0,$    
since $(x^{\psi(k)},y^{\psi(k)})$ and $(x^{\psi(k)},y^{\psi(k-1)})$ both converges towards $(x,y)$. We deduce that $\phi(x^k)^+ \rightarrow 0$ \cite[Ex.~2.5.4]{abbott}. 
\end{appendices}

\end{document}